\pdfoutput=1
\newif\ifpreprint%
\preprinttrue   

\ifpreprint
\documentclass[12pt]{article}

\usepackage{abstract} 

\usepackage[small,compact]{titlesec}
\usepackage{parskip}
\def\sectionfont{\sffamily\Large\bfseries\boldmath}
\def\subsectionfont{\sffamily\large\bfseries\boldmath}
\def\paragraphfont{\sffamily\normalsize\bfseries\boldmath}
\titleformat*{\section}{\sectionfont}
\titleformat*{\subsection}{\subsectionfont}
\titleformat*{\subsubsection}{\paragraphfont}
\titleformat*{\paragraph}{\paragraphfont}
\titleformat*{\subparagraph}{\paragraphfont}
\usepackage[small,labelfont={bf,sf}]{caption}  
\usepackage{fullpage,graphicx,psfrag,url}
\usepackage{multirow}
\usepackage[shortlabels]{enumitem}
\setlist{nolistsep}

\usepackage{booktabs}           
\usepackage{adjustbox} 

\newcounter{algorithmctr}[section]
\renewcommand{\thealgorithmctr}{\thesection.\arabic{algorithmctr}}

\newtheorem{theorem}{Theorem}
\newtheorem{lemma}{Lemma}
\newtheorem{assumption}{Assumption}
\newtheorem{definition}{Definition}

\makeatletter
\DeclareRobustCommand{\qed}{%
	\ifmmode 
	\else \leavevmode\unskip\penalty9999 \hbox{}\nobreak\hfill
	\fi
	\quad\hbox{\qedsymbol}}
\newcommand{\openbox}{\leavevmode
	\hbox to.77778em{%
		\hfil\vrule
		\vbox to.675em{\hrule width.6em\vfil\hrule}%
		\vrule\hfil}}
\newcommand{\qedsymbol}{\openbox}
\newenvironment{proof}[1][\proofname]{\par
	\normalfont
	\topsep6\p@\@plus6\p@ \trivlist
	\item[\hskip\labelsep\itshape
	#1.]\ignorespaces
}{%
	\qed\endtrivlist
}
\newcommand{\proofname}{Proof}
\makeatother

\usepackage[noend]{algpseudocode}
\usepackage[english]{babel}
\usepackage{algorithm}

\usepackage{mathtools,amsfonts}
\usepackage{xcolor}
\usepackage[xindy, acronyms]{glossaries}   

\else

\documentclass[sn-mathphys]{sn-jnl}       

\theoremstyle{thmstyleone}%
\newtheorem{theorem}{Theorem}
\newtheorem{lemma}[theorem]{Lemma}%

\theoremstyle{thmstyletwo}%

\theoremstyle{thmstylethree}%
\newtheorem{definition}{Definition}%

\fi

\usepackage{multicol}
\usepackage{subcaption}

\ifpreprint
\else
\fi

\ifpreprint
\title{\bfseries\sffamily An Efficient IPM Implementation for A Class of Nonsymmetric Cones}
\author{Yuwen Chen and Paul Goulart}
\fi

\usepackage{xcolor,calc}

\begin{document}

\ifpreprint
\maketitle

\begin{abstract}
	We present an efficient implementation of interior point methods for a family of nonsymmetric cones, including generalized power cones, power mean cones and relative entropy cones, by exploiting underlying low-rank and sparse properties of Hessians of homogeneous self-concordant barrier functions. We prove that the augmented linear system in our interior point method is sparse and quasi-definite, enabling the use of sparse LDL factorization with a dual scaling strategy for nonsymmetric cones. Numerical results show that our proposed implementation for nonsymmetric cones performs much faster than the state-of-art solvers for spare problems and scales well for large problems.
\end{abstract}
\fi

\section{Literature Review}
Convex optimization has been widely used in fields such as control engineering~\cite{Boyd94}, signal processing~\cite{Luo06} and finance~\cite{Cornuejols18}, and includes common problem classes such as linear programming (LP), quadratic programming (QP), second-order cone programming (QP) and semidefinite programming (SDP) problems, as well as other problems classes defined over more exotic cone constraints. Common modern approaches for general convex optimization include first-order operator-splitting methods~\cite{Odonoghue16,Garstka21}, and second-order interior-point methods (IPMs)~\cite{Nesterov98}, where the latter is preferred when high accuracy solutions are required. Interior point methods commonly employ logarithmically-homogeneous self-concordant barrier (LHSCB) functions, which apply penalties on convex inequality constraints defined over cones. The self-concordant parameter $\nu$ is related to the convergence speed of an IPM, where the total iteration number is known to be $\mathcal{O}(\sqrt{\nu}\ln(1/\epsilon))$ in theory. 

State-of-the-art conic interior-point solvers commonly employ an infeasible starting-point method where the initial iterate is chosen on the interior of any conic inequality constraints, with linear equalities not necessarily satisfied until the interior-point algorithm converges to optimality. Since general purpose optimization methods must allow for the possibility of infeasible problems being posed, IPMs require a mechanism that will work either by providing a solution (in the feasible case) or a certificate of infeasibility (in the infeasible case).  This requirement motivated the development of the homogeneous self-dual embedding (HSDE)~\cite{Ye94}, which augments the original problem to a new one that is always feasible. The new augmented problem is always solvable, and from its solution one can recover either a solution or certificate of infeasibility for the original problem. The HSDE method is very popular and is widely used in many solvers, e.g. SCS~\cite{Odonoghue16}, CVXOPT~\cite{cvxopt}, ECOS~\cite{Domahidi13} and Mosek~\cite{mosek}.

The great majority of research on efficient IPMs has focussed on problems with conic constraints that are self-dual (also called self-scaled), e.g. linear programming (LP), quadratic programming (QP), second order cone programming (SOCP) and semidefinite programming (SDP).  Cones that do \emph{not} possess the self-duality property are called \emph{nonsymmetric} and are a topic of more recent interest interior point methods~\cite{Nesterov12}. The extension of the homogeneous self-dual model for the nonsymmetric case was proposed in~\cite{Skajaa15} with a variant implemented in ECOS~\cite{Serrano15}, and then a complementary proof for $\mathcal{O}(\sqrt{\nu}\ln(1/\epsilon))$ convergence provided in~\cite{Papp17}. So far, the exponential cone and the power cone are the two most studied and are available in the commercial solver Mosek~\cite{mosek} and some open-source solvers like ECOS~\cite{Domahidi13} and Alfonso~\cite{alfonso}. 
It has also been shown that some existing conic optimization problems can be solved more efficiently by exploiting their special structure through the lens of nonsymmetric conic optimization, like sparse SDPs~\cite{Andersen10} and sum of squares (SOS) programs~\cite{Papp19}. A Matlab-based solver DDS~\cite{DDS} and a Julia-based solver Hypatia introduce some nonsymmetric cones that solve many optimization problems more efficiently than their extended formulations based on three-dimensional exponential cones or power cones. 

However, nonsymmetric cones are not self-scaled, so one can't compute Nesterov-Todd (NT) scaling points as is possible in the symmetric case. Instead, the dual scaling strategy is commonly used for nonsymmetric cones in solvers like ECOS and Hypatia. Recently, a primal-dual scaling algorithm motivated by~\cite{Tuncel01} was proposed in~\cite{Dahl21} and implemented in Mosek, which claims to be the fastest solver for exponential and power cones at present. Both the dual scaling and the primal-dual scaling, require conjugate gradients of conic barrier functions which usually don't have closed forms for nonsymmetric cones. The numerical implementation for computing conjugate gradients of some nonsymmetric cones is detailed in~\cite{Kapelevich22}.

The \emph{power cone} is a powerful tool to model various cones such as the $p$-norm cone, the power mean cone and the generalized power cone~\cite{Chares09}. The PhD thesis of Chares~\cite{Chares09} describes a $3$-self-concordant barrier function and conjectured a $(n+1)$-self-concordant barrier for $(n+1)$-dimensional power cone $\mathcal{K}_\alpha^{(n)}:=\left\{(x, z) \in \mathbb{R}_{+}^n \times \mathbb{R}: \prod_{i=1}^n x_i^{\alpha_i} \geq|z|\right\}$ of $\alpha \in \mathbb{R}^n_{\ge}$ and $\sum_{i=1}^n \alpha_i= 1$, which was finally validated in the context of a more general $(n,m)$-generalized power cone $\mathcal{K}_\alpha^{(n,m)}:=\left\{(x, z) \in \mathbb{R}_{+}^n \times \mathbb{R}^m: \prod_{i=1}^n x_i^{\alpha_i} \geq \|z \|_2 \right\}$~\cite{Roy22}. The \emph{relative entropy cone} can be regarded as the extension of three-dimensional exponential cone or second-order cone and is suitable to model many problems in circuit design, statistical learning and power control in communication systems, etc.~\cite{Chandrasekaran17}.


In this paper, our contribution are: $(1)$ to propose a sparse decomposition of generalized power cones, power mean cones, relative entropy cones that can be used in direct sparse LDL factorization; $(2)$ to show how to compute the conjugate barrier functions for these three nonsymmetric cones to measure the proximity to the central path. 

The sketch of this paper is given as follows: We introduce basic definitions of homogeneous self-dual embedding and interior point methods in Section~\ref{section-fundamentals}. Then, we propose a general structure that is suitable for sparse LDL factorization in Section~\ref{section:augmented-sparsity}. The proposed augmented sparse decomposition of Hessians for nonsymmetric cones are shown in Section~\ref{section-sparse-decomposition} and the centrality check for IPMs with nonsymmetric cones is detailed in Section~\ref{section-centrality-check}. In Section~\ref{section-experiments}, numerical results show that our proposed sparse implementation for nonsymmetric cones performs better than the state-of-art solvers. Finally, the conclusion is summarized in Section~\ref{section-conclusion}.

\paragraph*{Notation}
$\mathbb{R}^d$ is the $d$-dimensional space of reals and $x \in \mathbb{R}_{+}^d$ $(\mathbb{R}_{++}^d)$ denotes $x \in \mathbb{R}^d$ is nonnegative (positive) elementwise. $X \in \mathbb{S}^n$ denotes $X$ is a symmetric matrix in $\mathbb{R}^{n \times n}$. $X \in \mathbb{S}_{+}^n$ ($\mathbb{S}_{++}^n$) denotes $X$ is positive semidefinite (positive definite). Functions $f(\cdot)$ and $f^*(\cdot)$ consist of a convex conjugate pair.  The Hessian of $f$ at $x$ is denoted by $f''(x)$ and the 3rd-order directional derivative along direction $h$ by $f'''(x)[h] \in \mathbb{R}^{n \times n}$. The set of interior points of $\mathcal{C}$ is denoted by $\text{int}(\mathcal{C})$.  The inner product of vectors $x,y$ in the Euclidean space $\mathbb{R}^n$ is denoted by $\langle x,y \rangle$ and the Euclidean norm of $x$ by $\|x\|$.  The product set of $\mathcal{C}_1, \mathcal{C}_2$ is denoted $\mathcal{C} := \mathcal{C}_1 \times \mathcal{C}_2$, e.g.\ $(x,y) \in \mathcal{C}$ means $x \in \mathcal{C}_1, y \in \mathcal{C}_2$. The index set $\{1, 2, \dots, d\}$ is abbreviated as $[d]$. An $n$-dimensional vector with value $1$ for every entry is denoted by $\mathbf{1}^{n}$.  The indicator function $\delta_{i, j}$ is
\begin{align*}
	\delta_{i, j} = \left\{ 
	\begin{aligned}
		1, & \qquad i=j, \\ 0, & \qquad \text{otherwise}.
	\end{aligned}
	\right.
\end{align*}

\section{Problem Description}\label{section-fundamentals}

Throughout, we will consider the following primal-dual pair of conic optimization problems:

\begin{subequations}
	\begin{minipage}{0.40\textwidth}
		\begin{align}
		\begin{aligned}
			\min_{x,s} & \quad c^\top x \\
			\text{s.t.} & \quad Gx = h, \\
			& Ax + s = b, \ s \in \mathcal{K},
		\end{aligned}\label{problem-primal} \tag{$\mathcal{P}$}
		\end{align}
	\end{minipage}
	\begin{minipage}{0.1\textwidth}
		~
	\end{minipage}
	\begin{minipage}{0.40\textwidth}
		\begin{align}
		\begin{aligned}
			\max_{y,z} & \ -h^\top y - b^\top z \\
			\text{s.t.} & \quad A^\top z + G^\top y + c = 0, \\
			& z \in \mathcal{K}^*,
		\end{aligned}\label{problem-dual} \tag{$\mathcal{D}$}
		\end{align}
	\end{minipage} \hfill
	\\[2ex]
\end{subequations}
where $G \in \mathbb{R}^{p \times n}, A \in \mathbb{R}^{m \times n}, c \in \mathbb{R}^{n}, h \in \mathbb{R}^{p}, b \in \mathbb{R}^{m}$, $\mathcal{K}$ is a nonempty, closed convex cone and $\mathcal{K}^*$ is the dual cone of $\mathcal{K}$. We assume that $\mathcal{K}$ does not contain any straight lines, i.e.\ it is \emph{pointed}.  We call problem~\eqref{problem-primal} the \emph{primal} problem and~\eqref{problem-dual} the \emph{dual} problem, and will assume throughout that strong duality holds.  

The problem \eqref{problem-primal} and its dual \eqref{problem-dual} are quite general, and cover a wide range of standard problem types in numerical optimization.  If $\mathcal{K}$ is the nonnegative orthant, i.e.\ $\mathcal{K} = \mathbb{R}_{+}$, then \eqref{problem-primal} represents a linear program.   Likewise, if $\mathcal{K}$ is the second-order cone $\mathcal{K}_{\text{soc}}$ or the positive semidefinite cone $\mathcal{S}^n_+$, then \eqref{problem-primal} represents a second-order cone program (SOCP) or semidefinite program (SDP), respectively.   In each of these cases, or where $\mathcal{K}$ is a composition of cones of these types, the cone $\mathcal{K}$ is \emph{symmetric} and satisfies the self-dual property $\mathcal{K} = \mathcal{K}^*$.   

In this paper we focus on problems in which the cone $\mathcal{K}$ is \emph{non-symmetric} so that the self-dual property for $\mathcal{K}$ does \emph{not} hold, and will consider in detail a family of non-symmetric cones for which we can exploit special structure within an interior point method.   In the remainder of this section, we sketch the basic structure of an interior point method for such cones, and subsequently consider in detail the issue of sparsity exploitation for our cones of interest.

\subsubsection*{Logarithmically homogeneous self-concordant barrier (LHSCB) functions}
A function $f: \mathbb{R}^n \to \mathbb{R}$ is called $\nu-$LHSCB for a convex cone $\mathcal{K} \subset \mathbb{R}^n$, if it satisfies 
\begin{align} 
	\begin{aligned}
		&\vert \nabla^3 f(x)[r,r,r] \vert \le 2\left(\nabla^2 f(x)[r,r]\right)^{3/2} & \forall x \in \mathrm{int}\mathcal{K}, r \in \mathbb{R}^n, \\
		&f(\lambda x) = f(x) - \nu \ln(\lambda) & \forall x \in \mathrm{int}\mathcal{K}, \lambda > 0,
	\end{aligned}\label{self-concordant}
\end{align}
for some scalar $\nu \ge 1$.  We first recall several standard properties of such a barrier function; the reader is referred to~\cite{Nesterov97} for details.

The gradient $g$ of $f$ satisfies
\begin{align}
	\langle g(x),x \rangle = - \nu, \quad \forall x \in \mathrm{int}\mathcal{K}. \label{fundamental-nu-equality}
\end{align}
The convex conjugate $f^*$, defined as
\begin{align}
	f^*(y) := \max_{x \in \mathrm{int}\mathcal{K}} \{-\langle y,x \rangle - f(x)\}, \label{fundamental-conjugate-barrier-1}
\end{align}
is also $\nu-$LHSCB in the interior of $\mathcal{K}^*$, and we call it the \emph{conjugate barrier}. The gradient $g^*$ of $f^*$ is the solution of
\begin{align}
	g^*(y) := - \arg \max_{x \in \mathrm{int}\mathcal{K}} \{-\langle y,x \rangle - f(x)\}. \label{fundamental-conjugate-gradient}
\end{align} 
Likewise, we also call $f$ is the conjugate and $g$ is the conjugate gradient of $f^*$. From \eqref{fundamental-nu-equality}--\eqref{fundamental-conjugate-gradient} it can be verified that
\begin{align}
	f^*(y) = - \langle y, - g^*(y) \rangle - f(-g^*(y)) = -\nu - f(-g^*(y)), \ \forall y \in \mathrm{int}\mathcal{K}^*. \label{fundamental-conjugate-barrier-2}
\end{align}
Finally, we have 
\begin{align}
	-g^*(-g(x)) = x, \ \forall x \in \mathrm{int}\mathcal{K}, \quad -g(-g^*(y)) = y, \ \forall y \in \mathrm{int}\mathcal{K}^*. \label{fundamental-bilinear-map}
\end{align}

\subsection{Homogeneous self-dual IPM for nonsymmetric cones}\label{section-homogeneous-IPM}
We next introduce several basic building blocks for an IPM based on the homogeneous self-dual embedding (HSDE) for nonsymmetric cones.
In the simplified HSDE model the KKT condition for~\eqref{problem-primal} and~\eqref{problem-dual} are modified to
\begin{equation}
\begin{gathered}
	\begin{array}{c}
		\left[\begin{array}{l}
			0 \\
			0 \\
			s \\
			\kappa
		\end{array}\right]= \left[\begin{array}{cccc}
			0 & G^\top & A^\top & c \\
			-G & 0 & 0 & h \\
			-A & 0 & 0 & b \\
			-c^\top & -h^\top & -b^\top & 0
		\end{array}\right]\left[\begin{array}{l}
			x \\
			y \\
			z \\
			\tau
		\end{array}\right] 
	\end{array} \\[2ex]
		(s, z, \kappa, \tau) \in \mathcal{F} , x \in \mathbb{R}^n, \ y \in \mathbb{R}^p,
\end{gathered}\label{HSDE}
\end{equation}
where $\mathcal{F}:=\mathcal{K} \times \mathcal{K}^* \times \mathbb{R}_{+} \times \mathbb{R}_{+}$ and  the cone $\mathcal{K}$ may represent the composition of multiple standard cone types.    We refer to \eqref{HSDE} as an HSDE model for \eqref{problem-primal}--\eqref{problem-dual} since the problem of finding a feasible point for \eqref{HSDE} is dual to itself\footnote{NB: this does \emph{not} require that the cone $\mathcal{K}$ should be self-dual.}, its feasible points form a cone, and any feasible point $(x^*, y^*, z^*, s^*, \tau^*, \kappa^*)$ of \eqref{HSDE} can be used to extract either a solution to \eqref{problem-primal}--\eqref{problem-dual} or a certificate of infeasibility.   In particular, if $\tau^* > 0$ then \eqref{problem-primal}--\eqref{problem-dual} is feasible with solution $(x^*/\tau^*, y^*/\tau^*, z^*/\tau^*, s^*/\tau^*)$.    Our IPM approach amounts to a Newton-like method for finding a feasible point for \eqref{HSDE}.  

The optimality condition in a convex problem include a complementary slackness condition, which is replaced by a \emph{central path} condition in an IPM whose details depend on the scaling strategy employed. For symmetric cones, e.g.\ nonnegative, second-order and positive-semidefinite cones, we can exploit the self-scaled property~\cite{Nesterov97} that ensures a unique scaling point called the Nesterov-Todd (NT) scaling~\cite{Nesterov97} and is obtained from the combination of primal and dual variables. In the case of nonsymmetric cones, the dual scaling is a popular choice~\cite{Nesterov12,Skajaa15,Andersen10} and is implemented in ECOS~\cite{Serrano15}. A primal-dual scaling with a 3rd-order correction is also effective and is implemented in Mosek~\cite{Dahl21}. In this paper, we will employ the dual scaling strategy so that we can exploit the low-rank information or sparsity of Hessians of dual barrier functions directly. 


%

For the dual scaling strategy, given some initial point $(x^0, y^0, s^0, z^0, \tau^0, \kappa^0)$ with $(s^0,z^0) \in \mathcal{K} \times \mathcal{K}^*$,
the central path is parametrized by $\mu \in (0,1]$ and the gradient $g^*(z)$ of the dual barrier function $f^*(z)$, satisfying
\begin{align}
\begin{aligned}
	\left[\begin{array}{l}
		0 \\
		0 \\
		s \\
		\kappa
	\end{array}\right]= & \left[\begin{array}{cccc}
		0 & G^\top & A^\top & c \\
		-G & 0 & 0 & h \\
		-A & 0 & 0 & b \\
		-c^\top & -h^\top & -b^\top & 0
	\end{array}\right]\left[\begin{array}{l}
		x \\
		y \\
		z \\
		\tau
	\end{array}\right]+\mu\left[\begin{array}{l}
		q_x \\
		q_y \\
		q_z \\
		q_\tau
	\end{array}\right] \\
	& (s, z, \kappa, \tau) \in \mathcal{F}, \\
	& s = -\mu g^*(z), \quad \tau \kappa=\mu,
\end{aligned}\label{central-path}
\end{align}
where
\begin{align*}
	q_x &= -G^\top y^0 -A^\top z^0 -c \tau^0, & q_y &= G x^0 -h \tau^0, \\[1ex]
	q_\tau &= \kappa+c^\top x^0 +h^\top y^0+b^\top z^0, & q_z &= s^0+A x^0- b \tau^0.
\end{align*}
For an IPM with nonsymmetric cones, we require that every iterate should remain close to the central path. Standard metrics for this distance include the proximity measure in~\cite{Nesterov12,Skajaa15} and one defined in terms of shadow iterates~\cite{Tuncel01,Dahl21}.


\subsubsection{Computation of the KKT system}
In order to solve \eqref{HSDE} we adopt the widely-used affine-centering (predictor-corrector) framework due to its excellent practical performance. We provide implementation details for this method in Section~\ref{section-implementation}. Note that the step equations for both the predictor and the corrector steps can be written in the unified general form
\begin{align}
	\begin{aligned}
		& \left[\begin{array}{cccc}
			0 & G^\top & A^\top & c \\
			-G & 0 & 0 & h \\
			-A & 0 & 0 & b \\
			-c^\top & -h^\top & -b^\top & 0
		\end{array}\right]\left[\begin{array}{l}
			\Delta x \\
			\Delta y \\
			\Delta z \\
			\Delta {\tau}
		\end{array}\right]  - \left[\begin{array}{c}
			0 \\
			0 \\
			\Delta s \\
			\Delta {\kappa}
		\end{array}\right] = \left[\begin{array}{l}
			d_x \\
			d_y \\
			d_z \\
			d_\tau
		\end{array}\right],\\
		& \qquad \mu H^*(z)\Delta z + \Delta s = -d_s, \ \tau \Delta \kappa + \kappa \Delta \tau = -d_{\kappa},
	\end{aligned}\label{general-KKT}
\end{align}
where only the choice of right-hand side term $d_x,d_y,d_z,d_s,d_{\tau},d_{\kappa}$ differs between the affine and the centering step directions and a given iterate. $H^*(z)$ is the Hessian of the dual barrier $f^*(z), \forall z \in \mathrm{int}\mathcal{K}$.

We typically expect the data matrices $G$ and $A$ to be sparse, particularly for large problems, so we adopt the same approach as \cite{cvxopt,Domahidi13} to preserve this sparsity while reducing the system to one in for which we can perform a symmetric linear solve. First, we substitute the last two equations into the first one to obtain 
\begin{align*}
	\begin{aligned}
		& \left[\begin{array}{cccc}
			0 & G^\top & A^\top & c \\
			G & 0 & 0 & -h \\
			A & 0 & -\mu H^*(z) & -b \\
			-c^\top & -h^\top & -b^\top & \frac{\kappa}{\tau}
		\end{array}\right]\left[\begin{array}{l}
			\Delta x \\
			\Delta y \\
			\Delta z \\
			\Delta {\tau}
		\end{array}\right] = \left[\begin{array}{c}
			d_x \\
			-d_y \\
			d_s - d_z\\
			d_\tau - d_\kappa/\tau
		\end{array}\right].
	\end{aligned}
\end{align*}
This can then be solved by solving a symmetric linear system twice with different right-hand sides, i.e.
\begin{align}
	K\left[\begin{array}{l}
		 x_1 \\
		 y_1 \\
		 z_1
	\end{array}\right]  = \left[\begin{array}{c}
		d_x \\
		-d_y \\
		d_s - d_z
	\end{array}\right], \qquad
	K \left[\begin{array}{l}
		 x_2 \\
		 y_2 \\
		 z_2
	\end{array}\right]  = \left[\begin{array}{r}
		-c \\
		h \\
		b
	\end{array}\right],\label{reduced-KKT-1}
\end{align}
where $K$ is defined as
\begin{align}
	K := \left[\begin{array}{ccc}
		0 & G^\top & A^\top \\
		G & 0 & 0  \\
		A & 0 & -H_s \\
	\end{array}\right] \label{definition-K}
\end{align} 
with $H_s = \mu H^*(z)$. Only the lower right-hand block $H_s$, which may be dense in general, would be different if we had instead employed a primal-dual scaling strategy~\cite{Dahl21}. 
Finally, we obtain the solution of~\eqref{general-KKT} by first solving for $\Delta \tau$ as 
\[
	\Delta \tau=\frac{d_\tau-d_\kappa / \tau+ c^\top x_1 + h^\top y_1 + b^\top z_1}{\kappa / \tau - c^\top x_2 - h^\top y_2 - b^\top z_2}, \\
\]
followed by
\begin{align*}
	\begin{aligned}[t]
	\Delta x &=  x_1 + \Delta \tau x_2,\\
	\Delta y &=  y_1 + \Delta \tau y_2,\\
	\Delta z &=  z_1 + \Delta \tau z_2,
	\end{aligned} ~~~~~
	\begin{aligned}[t]
	\Delta s &= -d_s-H_s\Delta z, \\ 
	\Delta \kappa&=-\left(d_\kappa+\kappa \Delta \tau\right) / \tau.
	\end{aligned}
\end{align*}
Observe that we only need to factor the matrix $K$ once at each iteration before performing the two linear solves in~\eqref{reduced-KKT-1}. We utilize sparse LDL factorization in~\cite{Davis05} with static ordering to factorize the matrix $K$. We add a small regularization term to the diagonal of $K$ so that it becomes symmetric quasidefinite and the LDL factorization always exists with diagonal $D$ even after any symmetric permutation $P$~\cite{Vanderbei95}, i.e. $K = PLDL^\top P^\top$.



\section{Augmented Sparsity}\label{section:augmented-sparsity}
We will consider a class of nonsymmetric cones for which the term $H_s$ that appears as the lower right-hand block in \eqref{definition-K} is the sum of a sparse matrix plus a few low-rank dense terms, and satisfies a certain quasi-definiteness condition.  
\begin{definition}\label{definition:augmented-sparse}
	Suppose the matrix $H$ is of the form
	\begin{align}
		H := D + \sum_{i=1}^{n_1}u_iu_i^\top - \sum_{j=1}^{n_2}v_iv_i^\top, \label{augmented-sparse}
	\end{align}
	where $D \in \mathbb{S}^{n \times n}$ is sparse and $n_1,n_2 \ll n$. We will call $H$ \emph{augmented-sparse} if $D - \sum_{j=1}^{n_2}v_iv_i^\top \succ 0$.   The corresponding \emph{augmented sparse matrix} is 
	\begin{align*}
		H_{\mathrm{aug}} := \left[\begin{array}{ccc}
			D & V & U\\
			V^\top & I_{n_2} &\\
			U^\top & & -I_{n_1}
		\end{array}\right]
	\end{align*}
	with $V = [v_1, \dots, v_{n_2}], U = [u_1, \dots, u_{n_1}]$.
\end{definition}

\begin{lemma}
If $H$ is augmented-sparse then $H_{\mathrm{aug}}$ is quasidefinite.
\end{lemma}
\begin{proof}
 A sufficient condition is for the upper left 2$\times$2 block of $H$ 
 to be positive definite \cite{Vanderbei95}.  This follows immediately since the Schur complement $D - VV^T$ is positive definite by definition. 
\end{proof}
If $H_s$ in~\eqref{definition-K} is augmented-sparse, we can solve~\eqref{reduced-KKT-1} via an equivalent linear system. For example, the last row of the left equation in~\eqref{reduced-KKT-1}, i.e. $Ax - H_s z_1 = d_s - d_z$ becomes
\begin{align}
	\left[\begin{array}{c}
		A \\ 0 \\ 0
	\end{array}\right] \Delta x -
	\underbrace{\left[\begin{array}{ccc}
			D & V & U\\
			V^\top & I_{n_2} &\\
			U^\top & & -I_{n_1}
		\end{array}\right]}_{H_{s,\mathrm{aug}}}
	\left[\begin{array}{c}
		\Delta z_1 \\ t_v \\ t_u
	\end{array}\right]
	= \left[\begin{array}{c}
		d_s - d_z\\ 0 \\ 0
	\end{array}\right],\label{rank-expansion}
\end{align}
which is both larger and sparser than the equivalent term in~\eqref{reduced-KKT-1}. The same applies to the last row of the right equation in~\eqref{reduced-KKT-1}. 
For~\eqref{rank-expansion}, the number of nonzero entries in $H_{s,\mathrm{aug}}$ is $[\mathrm{nnz}(D) + n(n_1+n_2+2)]$, which is much smaller than the number of entries in a dense $H$ if $n_1,n_2 \ll n$.  We therefore expect that an  LDL factorization of $H_{\mathrm{aug}}$ will have significantly sparser factors than would be obtained by direct factorization of $H$. Such a property has been exploited for SOCPs in ECOS under the NT (primal-dual) scaling strategy~\cite{Domahidi13}. We will show how this approached can also be utilized in the dual scaling strategy for nonsymmetric cones.

\section{Sparse Decomposition of $H_s$ in Nonsymmetric Cones}\label{section-sparse-decomposition}
In this section, we describe several non-symmetric cones that can be shown to have \textit{augmented-sparse} structure in the sense of Definition~\ref{definition:augmented-sparse}. We introduce fundamental definitions and barrier functions for these nonsymmetric cones in Section~\ref{subsection:intro-nonsymmetric cones}, and then detail the underlying augmented-sparse structures in Section~\ref{subsection:generalized-power}-\ref{subsection:relative-entropy-cone}, respectively.

\subsection{Three nonsymmetric cones}\label{subsection:intro-nonsymmetric cones}
We consider three nonsymmetric cones that can exploit sparse structure: generalized power cones, power mean cones and relative entropy cones. They are defined as follows:

\begin{definition}[Generalized Power Cone]
The \emph{generalized power cone} is parametrized by $\alpha \in \mathbb{R}_{++}^{d_1}$ such that $\sum_{i \in [d_1]} \alpha_i=1$ and is defined as 
\begin{subequations}
	\begin{align}
	\mathcal{K}_{\mathrm{gpow}(\alpha,d_1,d_2)} &= \left\{(u, w) \in \mathbb{R}_{+}^{d_1} \times \mathbb{R}^{d_2}: \prod_{i \in [d_1]} u_i^{\alpha_i} \geq \|w\| \right\}.
\intertext{Its dual cone is}
	\mathcal{K}_{\mathrm{gpow}(\alpha,d_1,d_2)}^* &= \left\{(u, w) \in \mathbb{R}_{+}^{d_1} \times \mathbb{R}^{d_2}: \prod_{i \in [d_1]}\left(\frac{u_i}{\alpha_i}\right)^{\alpha_i} \geq \|w\| \right\}.
\end{align}
\end{subequations}
\end{definition}

\begin{definition}[Power Mean Cone]
 The \emph{power mean cone} is parametrized by $\alpha \in \mathbb{R}_{++}^{d_1}$ such that $\sum_{i \in [d]} \alpha_i=1$ and is defined as 
\begin{subequations}
\begin{align}
	\mathcal{K}_{\mathrm{powm}(\alpha,d)} &= \left\{(u, w) \in \mathbb{R}_{+}^{d} \times \mathbb{R}: \prod_{i \in [d]} u_i^{\alpha_i} \geq w \right\}.
\intertext{Its dual cone is}
	\mathcal{K}_{\mathrm{powm}(\alpha,d)}^* &= \left\{(u, w) \in \mathbb{R}_{+}^{d} \times \mathbb{R}_-: \prod_{i \in [d]}\left(\frac{u_i}{\alpha_i}\right)^{\alpha_i} \geq -w \right\}.
\end{align}
\end{subequations}
\end{definition}

Note that the geometric mean cone~\cite{Hypatia} is a special case of the  power mean cone where $\alpha_1 = \dots = \alpha_d = \frac{1}{d}$.

\begin{definition}[Relative Entropy Cone]
The \emph{relative entropy cone} is defined as 
\begin{subequations}
\begin{align}
	\mathcal{K}_{\mathrm{rel}} &= \text{cl}\left\{(u,v,w) \in \mathbb{R} \times \mathbb{R}_{++}^d \times \mathbb{R}_{++}^d : u \ge \sum_{i \in [d]} w_i \ln\left(\frac{w_i}{v_i}\right) \right\}.\label{entropy-cone}
\intertext{Its dual cone is}
	\mathcal{K}_{\mathrm{rel}}^* &= \text{cl}\left\{(u,v,w) \in \mathbb{R}_{++} \times \mathbb{R}_{++}^d \times \mathbb{R}^d : w_i \ge u\left(\ln\left(\frac{u}{v_i}\right)-1\right), \forall i \in [d] \right\}.
\end{align}
\end{subequations}
\end{definition}
Note that a generalized power cone reduce to a second-order cone when $d_1=1$ and it becomes self-dual, i.e.\ $\mathcal{K}_{\text {gpow }(1,1,d_2)} = \mathcal{K}_{\text {gpow }(1,1,d_2)}^*$, while it looks similar to a power mean cone when $d_2=1$. This motivates us to consider whether we can exploit sparsity of these nonsymmetric cones using an approach similar to the one that was  described for second-order cones in~\cite{Domahidi13}. The answer is \emph{yes} if we choose the dual scaling strategy for these cones in a IPM.  Meanwhile, the dual of the relative entropy cone has a sparse Hessian that we can exploit and which also fits our Definition~\ref{definition:augmented-sparse}. 

We therefore choose the dual scaling strategy in our IPM throughout the paper. We consequently must define $\nu$-LHSCB conjugate barrier functions for each of the cones $\mathcal{K}_{\text {gpow }(\alpha,d_1,d_2)}^*$, $\mathcal{K}_{\text {powm}(\alpha,d)}^*$ and $\mathcal{K}_{\mathrm{rel}}^*$. The results are summarized in the next theorem.
\begin{theorem}\label{theorem-barrier}
	For the dual cones of generalized power cones, power mean cones and relative entropy cones:
	\begin{enumerate}[i)]
		\item The function
		\begin{align}
			f_{\text{gpow}}^*(u, w)=-\ln \left(\prod_{i \in [d_1]} \left(\frac{u_i}{\alpha_i}\right)^{2 \alpha_i}- \|w\|^2\right)-\sum_{i \in [d_1]}\left(1-\alpha_i \right) \ln \left(\frac{u_i}{\alpha_i}\right) \label{dual-barrier-genpow}
		\end{align}
		is a $(d_1+1)$-LHSCB function of $\mathcal{K}^*_{\text {gpow }(\alpha,d_1,d_2)}$ where $\alpha, u \in \mathbb{R}_{++}^{d_1}, w \in \mathbb{R}^{d_2}$. 
		\bigskip
		\item The function
		\begin{align}
			f_{\mathrm{powm}}^*(u, w)=-\ln \left(\prod_{i \in [d]} \left(\frac{u_i}{\alpha_i}\right)^{\alpha_i}+ w\right)-\sum_{i \in [d]} (1-\alpha_i)\ln \left(\frac{u_i}{\alpha_i}\right) - \ln(-w) \label{dual-barrier-powm}
		\end{align}
		is a $(d+1)$-LHSCB function of $\mathcal{K}^*_{\mathrm{powm}(\alpha,d)}$ where $\alpha, u \in \mathbb{R}_{++}^{d}, w < 0$. 
		\bigskip
		\item 	 The function
		\begin{align}
			f_{\mathrm{rel}}^*(u,v,w) = -\sum_{i \in [d]}\left(w_i - u\ln\left(\frac{u}{v_i}\right)+u\right) - d\ln(u) - \sum_{i \in [d]}\ln(v_i) \label{dual-barrier-rel}
		\end{align}
		is a $3d$-LHSCB function of $\mathcal{K}^*_{\mathrm{rel}}$ where $u > 0$ and $v \in \mathbb{R}_{++}^{d},w \in \mathbb{R}^{d}$.
	\end{enumerate}
\end{theorem}

\begin{proof}
See Appendix~\ref{appendix-proof}. 
\end{proof}

\subsection{Sparsity exploitation}
For the dual-scaling strategy at iteration $k$, the scaling matrix $H_s$ is set to $H_s^k = \mu^k H^*(z^k)$, where $\mu^k > 0$ is the centering parameter and $H^*(z^k)$ is the Hessian of the dual barrier functions at $z^k \in \mathcal{K}^*$. 
Hence, \textit{we can exploit the augmented-sparse structure of the scaling matrix $H_s$ as long as $H^*(z^k)$ is augmented-sparse}. For the remainder of this Section, we will show that the Hessians $H^*(z)$ of each of the dual barrier functions defined in Theorem~\ref{theorem-barrier} are all augmented-sparse, and thus so are the related scaling matrices $H_s$.

\subsubsection{Generalized power cone}\label{subsection:generalized-power}
We start by establishing the augmented-sparse property for generalized power cones:
\begin{theorem}\label{theorem-genpow-KKT}
	The Hessian of the barrier function~\eqref{dual-barrier-genpow} for the generalized power cone satisfies Definition~\ref{definition:augmented-sparse} with $n_1 = 1, n_2 = 2$, i.e.
	\begin{align*}
		H^*(z) = D + pp^\top - qq^\top - rr^\top,
	\end{align*}
	where $z=:(u,w)$ and $D - qq^\top - rr^\top \succ 0$. The parameters $D, p, q, r$ are given by
	\begin{subequations}\label{generalized-power-mean-param}
	\begin{align}
		\begin{array}{c}
			D = \left[\begin{array}{ccc|c}
				\ddots & & & \\ 
				& \frac{\tau_i \varphi}{\zeta u_i} + \frac{1-\alpha_i}{u_i^2} & & \\
				& & \ddots & \\
				\hline
				& & & \frac{2}{\zeta} \cdot I_{d_2}
			\end{array}\right], 
			\ p = \left[\begin{array}{c}
				p_0 \cdot \frac{\tau}{\zeta}\\ p_1 \cdot \frac{w}{\zeta}
			\end{array}\right],
			\ q = \left[\begin{array}{c}
				q_0 \cdot \frac{\tau}{\zeta} \\ 0
			\end{array}\right], 
			\ r = \left[\begin{array}{c}
				0 \\ r_1 \cdot \frac{w}{\zeta}
			\end{array}\right], \\
		\end{array}
	\end{align}
	with 
	\begin{equation}
		\begin{aligned}[t]
			p_0 &= \sqrt{\frac{\varphi(\varphi+\|w\|^2)}{2}}, 
			&&& p_1 &= 
			-2\sqrt{\frac{2\varphi}{\varphi+\|w\|^2}}, \\[1ex]
			q_0 &= \sqrt{\frac{\zeta \varphi}{2}},
			&&& r_1 &= 2\sqrt{\frac{\zeta}{\varphi+\|w\|^2}},
		\end{aligned}
	\end{equation}
	\end{subequations}

	where $\varphi=\prod_{i \in [d_1]} \left(\frac{u_i}{\alpha_i}\right)^{2 \alpha_i}, \tau_i=\frac{2 \alpha_i}{u_i}, \forall i \in [d_1]$, and $\zeta=\prod_{i \in [d_1]} \left(\frac{u_i}{\alpha_i}\right)^{2 \alpha_i}- \|w\|^2$.
\end{theorem}

\begin{proof}[Proof of Theorem~\ref{theorem-genpow-KKT}]
	The gradient and Hessians for $f_{\text{gpow}}^*(u, w)$~\eqref{dual-barrier-genpow} are given by
	\begin{align*}
		\begin{aligned}
			\nabla_{u_i} f^* &=-\frac{\tau_i \varphi}{\zeta}-\frac{1-\alpha_i}{u_i}, \quad \forall i \in [d_1], \\
			\nabla_{w_i} f^* &=\frac{2 w_i}{\zeta}, \quad \forall i \in [d_2], \\
			\nabla_{u_i, u_j}^2 f^* &=\frac{\tau_i \tau_j \varphi}{\zeta}\left(\frac{\varphi}{\zeta}-1\right)+\delta_{i, j}\left(\frac{\tau_i \varphi}{\zeta u_i}+\frac{1-\alpha_i}{u_i^2}\right), \quad \forall i,j \in [d_1], \\
			\nabla_{u_i, w_j}^2 f^* &=-\frac{2 \tau_i \varphi w_j}{\zeta^2}, \quad \forall i \in [d_1], \forall j \in [d_2],\\
			\nabla_{w_i, w_j}^2 f^* &=\frac{4 w_i w_j}{\zeta^2}+\delta_{i, j}\frac{2}{\zeta}, \quad \forall i,j \in [d_2].
		\end{aligned}
	\end{align*}
	In order to satisfy Definition~\ref{definition:augmented-sparse}, 
	$p_0,p_1,q_0,r_1$ need to satisfy the conditions
	\begin{align*}
		p_0^2 - q_0^2 = \varphi(\varphi - \zeta)= \varphi \|w\|^2, \quad p_0 p_1 = -2\varphi, \quad p_1^2 - r_1^2 = 4,
	\end{align*}
	to match the coefficients in the Hessian. Given the parameter choices in~\eqref{generalized-power-mean-param}, we find $D - qq^\top - rr^\top$ is $2 \times 2$ block diagonal and the condition $D - qq^\top - rr^\top \succ 0$ becomes
	\begin{align*}
		& D_1 - q_0^2 \left(\frac{\tau}{\zeta}\right) \cdot \left(\frac{\tau}{\zeta}\right)^\top \succ 0 \text{ and } D_2 - r_1^2\left(\frac{w}{\zeta}\right) \cdot \left(\frac{w}{\zeta}\right)^\top \succ 0,
	\end{align*}
	which is equivalent to
	\begin{subequations}
		\begin{gather}
			1 - \left(\frac{\tau}{\zeta}\right)^\top D_1^{-1} \left(\frac{\tau}{\zeta}\right) = 1 - \sum_{i \in [d_1]}\frac{q_0^2 \cdot \frac{\tau_i^2}{\zeta^2}}{ \frac{\tau_i \varphi}{\zeta u_i} + \frac{1-\alpha_i}{u_i^2}} = 1 - \sum_{i \in [d_1]}\frac{4\alpha_i^2 q_0^2 }{ \zeta(2\alpha_i \varphi + (1-\alpha_i)\zeta)} > 0, \label{theorem2-proof-eq-1}
\intertext{and}
			1 - \left(\frac{w}{\zeta}\right)^\top D_2^{-1} \left(\frac{w}{\zeta}\right) = 1 - r_1^2 \frac{\|w\|^2}{\zeta^2} \cdot \frac{\zeta}{2} > 0, \label{theorem2-proof-eq-2}
		\end{gather}
	\end{subequations}
	by Schur complement.
	If we set $q_0 = \sqrt{\frac{\zeta \varphi}{2}}, r_1 = 2\sqrt{\frac{\zeta}{\varphi+\|w\|^2}}$,~\eqref{theorem2-proof-eq-1} is satisfied by
	\begin{align*}
		1 - \sum_{i \in [d_1]}\frac{4 \alpha_i^2 q_0^2}{\zeta(2 \alpha_i \varphi + (1-\alpha_i)\zeta)} > 1 - \sum_{i \in [d_1]}\frac{4 \alpha_i^2 q_0^2}{2 \zeta \alpha_i \varphi} = 1 - \frac{2 q_0^2}{\zeta \varphi} = 0,
	\end{align*}
	and~\eqref{theorem2-proof-eq-2} is satisfied by
	\begin{align*}
		1 - r_1^2 \frac{\|w\|^2}{\zeta^2} \cdot \frac{\zeta}{2} = 1 - \frac{2 \|w\|^2}{\varphi + \|w\|^2} = \frac{2 \zeta}{\varphi + \|w\|^2} > 0.
	\end{align*}
	Hence, $D - qq^\top - rr^\top \succ 0$.
\end{proof}
Since the scaling matrix $H_s = \mu H^*(z)$ is augmented-sparse, we can employ an expanded sparse linear system
\begin{align}
	\left[\begin{array}{c}
		A \\ 0 \\ 0 \\ 0
	\end{array}\right] \Delta x
	- \underbrace{\mu \left[\begin{array}{cccc}
		D & q & r & p\\
		q^\top & 1 & 0 & 0\\
		r^\top & 0 & 1 & 0\\
		p^\top & 0 & 0 & -1
	\end{array}\right]}_{K_{\text{aug}}}
	\left[\begin{array}{c}
		\Delta z_1 \\ t_q \\ t_r \\ t_p
	\end{array}\right]
	= \left[\begin{array}{c}
		d_s - d_z\\ 0 \\ 0 \\ 0
	\end{array}\right],\label{rank3-expansion}
\end{align}
in our interior point step equation~\eqref{reduced-KKT-1}, 
where $K_{\text{aug}}$ is quasi-definite and strongly factorizable. Instead of~\eqref{rank3-expansion}, we solve the  linear system
\begin{align}
	\left[\begin{array}{c}
		A \\ 0 \\ 0 \\ 0
	\end{array}\right] \Delta x
	- \left[\begin{array}{cccc}
			\mu D & \sqrt{\mu} q & \sqrt{\mu}r & \sqrt{\mu}p\\
			\sqrt{\mu}q^\top & 1 & 0 & 0\\
			\sqrt{\mu}r^\top & 0 & 1 & 0\\
			\sqrt{\mu}p^\top & 0 & 0 & -1
		\end{array}\right]
	\left[\begin{array}{c}
		\Delta z_1 \\ t'_q \\ t'_r \\ t'_p
	\end{array}\right]
	= \left[\begin{array}{c}
		d_s - d_z\\ 0 \\ 0 \\ 0
	\end{array}\right]\label{rank3-expansion-stable},
\end{align}
which yields the same $\Delta x, \Delta z_1$ as~\eqref{rank3-expansion}. During direct factorization we find that the pivoting for~\eqref{rank3-expansion-stable} is more numerically stable than for~\eqref{rank3-expansion} since the magnitude of diagonal terms of $K_{\mathrm{aug}}$ increases from $\mu$ to $1$ and the ratio of off-diagonal to diagonal terms also decreases, e.g. $q$ to $\sqrt{\mu}q$.
In addition, since half of entries in $q,r$ are zeros, the memory required for $q,r$ together is the same as $p$ in the sparse data form.

\subsubsection{Power mean cone}\label{subsection:power-mean}
We next establish the augmented-sparse property for the power mean cone:
\begin{theorem}\label{theorem-power-mean}
	The Hessian of the barrier function~\eqref{dual-barrier-powm} for the power mean cone satisfies Definition~\ref{definition:augmented-sparse} with $n_1 = 1, n_2 = 2$, i.e.\
	\begin{align*}
		H^*(z) = D + pp^\top - qq^\top - rr^\top,
	\end{align*}
	where $z=:(u,w)$ and $D - qq^\top - rr^\top \succ 0$. The parameters $D, p, q, r$ are given by
	\begin{subequations}\label{power-mean-param}
	\begin{align}
		\begin{array}{c}
			D = \left[\begin{array}{ccc|c}
				\ddots & & & \\ 
				& \frac{\tau_i \varphi}{u_i} + \frac{1-\alpha_i}{u_i^2} & & \\
				& & \ddots & \\
				\hline
				& & & \theta
			\end{array}\right], 
			\ p = \left[\begin{array}{c}
				p_0 \cdot \tau\\ p_1 \cdot \frac{1}{\zeta}
			\end{array}\right],
			\ q = \left[\begin{array}{c}
				q_0 \cdot \tau \\ 0
			\end{array}\right], 
			\ r = \left[\begin{array}{c}
				0 \\ r_1 \cdot \frac{1}{\zeta}
			\end{array}\right],
		\end{array}
	\end{align}
	with 
	\begin{equation}
	\begin{aligned}[t]
		\theta &= 1 + \frac{1}{w^2}, 
		&&& p_0 &= \varphi, 
		&&& p_1 &= 1, &&& \\[1ex]
		&
		&&& q_0 &= \sqrt{\zeta \varphi},
		&&& r_1 &= \zeta,
	\end{aligned}
	\end{equation}
\end{subequations}
	where we define $\varphi=\prod_{i \in [d]} \left(\frac{u_i}{\alpha_i}\right)^{\alpha_i}$, $\zeta=\varphi+ w$ and $\tau \in \mathbb{R}^d$ with $\tau_i=\frac{\alpha_i}{u_i \zeta}, \forall i \in [d]$.
\end{theorem}
\begin{proof}[Proof of Theorem~\ref{theorem-power-mean}]
	The gradient of~\eqref{dual-barrier-powm} is
	\begin{align*}
		\nabla_{u_i} f^*=-\tau_i \varphi-\frac{1-\alpha_i}{u_i}, \ \forall i \in [d], \quad \nabla_{w} f^* = -\frac{1}{\zeta} - \frac{1}{w}, 
	\end{align*}
	and the corresponding Hessian is
	\begin{align}
	\begin{aligned}
				\nabla_{u_i, u_j}^2 f^* &= -\varphi w \tau_i \tau_j+\delta_{i, j}\left(\frac{\tau_i \varphi}{u_i}+\frac{1-\alpha_i}{u_i^2}\right), \quad \forall i,j \in [d],\\
		\nabla_{u_i, w}^2 f^* &=\frac{\tau_i \varphi}{\zeta}, \quad \forall i \in [d],\\
		\nabla_{w, w}^2 f^* &=\frac{1}{\zeta^2} + \frac{1}{w^2}.
	\end{aligned}\label{pown-hessian}
	\end{align}
	Our choices for $\theta,p_0,p_1,q_0,r_1$ satisfy
	\begin{align*}
		p_0^2 - q_0^2 = \varphi(\varphi - \zeta)= -\varphi w, \quad p_0 p_1 = \varphi, \quad \theta + \frac{p_1^2}{\zeta^2} - \frac{r_1^2}{\zeta^2} = \frac{1}{\zeta^2} + \frac{1}{w^2},
	\end{align*}
	which match coefficients in the dual Hessian~\eqref{pown-hessian}. Since $D - qq^\top - rr^\top$ is $2 \times 2$ block diagonal, the condition $D - qq^\top - rr^\top \succ 0$ is equivalent to $\theta-\frac{r_1^2}{\zeta^2} = \frac{1}{w^2} > 0 $ and
	\begin{align*}
		D_1 - q_0^2 \cdot \tau \tau^\top \succ 0, \iff 1 - \tau^\top D_1^{-1} \tau > 0,
	\end{align*}
	which can be shown by writing
	\begin{align*}
		1 - \tau^\top D_1^{-1} \tau & = 1 - \sum_{i \in [d]}\frac{q_0^2 \tau_i^2}{ \frac{\tau_i \varphi}{u_i} + \frac{1-\alpha_i}{u_i^2}} = 1 - \sum_{i \in [d]}\frac{\alpha_i^2 q_0^2 }{ \zeta(\alpha_i \varphi + (1-\alpha_i)\zeta)} \\
		& > 1 - \sum_{i \in [d]}\frac{\alpha_i^2 q_0^2 }{ \zeta \alpha_i \varphi} = 1 - \frac{q_0^2 }{ \zeta \varphi} = 0.
	\end{align*}
	Hence, $D - qq^\top - rr^\top \succ 0$.
\end{proof}
Note that the augmented-sparse structure of power mean cones is similar to the case of generalized power cones in Theorem~\ref{theorem-power-mean}, where both are composed of a diagonal matrix plus a rank-$3$ update and the memory requirement for the rank-$3$ update is equivalent to a rank-$2$ update due to the presence of zeros in $q,r$. Therefore, we can use an expanded sparse linear system as~\eqref{rank3-expansion} with a numerically stable variant similar to~\eqref{rank3-expansion-stable} for interior point step equations for the power mean cone.

\subsubsection{Relative entropy cone}\label{subsection:relative-entropy-cone}
The relative entropy cone has a slightly different structure than generalized power cones and power mean cones, but it still satisfies the augmented-sparse property as in Definition~\ref{definition:augmented-sparse}.
\begin{theorem}\label{theorem-rel-entropy}
	The Hessian of the barrier function~\eqref{dual-barrier-rel} for the relative entropy cone satisfies Definition~\ref{definition:augmented-sparse} with $n_1 = n_2 = 0$, i.e.\
	\begin{align*}
		\begin{array}{c}
			H^*(z) = \left[\begin{array}{ccc}
				\nabla_{u, u}^2 f^* & \nabla_{u, v}^2 f^* & \nabla_{u, w}^2 f^*\\ 
				\nabla_{v,u}^2 f^* & D_{v,v} & D_{v,w} \\
				\nabla_{w,u}^2 f^* & D_{w,v} & D_{w,w}
			\end{array}\right], 
		\end{array}
	\end{align*}
	where $z=:(u,v,w)$ and $H^*(z) \succ 0$ sparse for $d \gg 1$.
	The lower right $2\times 2$ blocks of $H^*(z)$ are given by the diagonal matrices 
	\begin{equation}
	\begin{aligned}
	D_{v,v} &= \mathrm{Diag}\left(\left[\frac{u(\gamma_1+u)}{\gamma_1^2v_1^2} + \frac{1}{v_1^2}, \dots, \frac{u(\gamma_d+u)}{\gamma_d^2v_d^2} + \frac{1}{v_d^2}\right]\right), \\[1ex]
	 D_{w,w} &= \mathrm{Diag}\left(\left[\frac{1}{\gamma_1^2},\dots, \frac{1}{\gamma_d^2}\right]\right), \\[1ex]
	 D_{v,w} = D_{w,v} &= \mathrm{Diag}\left(\left[\frac{u}{\gamma_1^2v_1}, \dots, \frac{u}{\gamma_d^2v_d}\right]\right), 
	\end{aligned}
	\end{equation}
	with $\gamma_i = w_i - u\ln\left(\frac{u}{v_i}\right)+u > 0, \forall i \in [d]$.
\end{theorem}
\begin{proof}[Proof of Theorem~\ref{theorem-rel-entropy}]
	The proof is straightforward. The gradient of~\eqref{dual-barrier-rel} is
	\begin{align*}
		\begin{aligned}
			\nabla_{u} f^* &= \sum_{i \in [d]} \frac{\tau_i}{\gamma_i} - \frac{d}{u}, \\
			\nabla_{v_i} f^* &= -\frac{u}{\gamma_i v_i} - \frac{1}{v_i}, \quad \forall i \in [d],\\
			\nabla_{w_i} f^* &= -\frac{1}{\gamma_i}, \quad \forall i \in [d],
		\end{aligned}
	\end{align*}
	where Let $\tau_i = \ln\left(\frac{u}{v_i}\right), \forall i \in [d]$.
	The Hessian of~\eqref{dual-barrier-rel} is
	\begin{align*}
		\begin{aligned}
			\nabla_{u, u}^2 f^* &= \frac{d}{u^2} + \sum_{i \in [d]}\frac{1}{u\gamma_i} + \sum_{i \in [d]} \left(\frac{\tau_i}{\gamma_i}\right)^2, \\
			\nabla_{v_i,u}^2f^*  &= \nabla_{u, v_i}^2 f^* = -\frac{1}{\gamma_i v_i} - \frac{u \tau_i}{\gamma_i^2v_i}, \quad \forall i \in [d],\\
			\nabla_{w_i,u}^2 f^* &= \nabla_{u, w_i}^2 f^* = -\frac{\tau_i}{\gamma_i^2}, \quad \forall i \in [d],\\
			\nabla_{v_i, v_j}^2 f^* &= \delta_{i, j}\left(\frac{u(\gamma_i+u)}{\gamma_i^2v_i^2} + \frac{1}{v_i^2}\right), \quad \forall i,j \in [d],\\
			\nabla_{w_i, w_j}^2 f^* &= \delta_{i, j}\frac{1}{\gamma_i^2}, \quad \forall i,j \in [d],\\
			\nabla_{v_i, w_j}^2 f^* &= \delta_{i, j}\frac{u}{\gamma_i^2v_i}, \quad \forall i,j \in [d].
		\end{aligned}
	\end{align*}
\end{proof}

\section{Functional Proximity Measures}\label{section-centrality-check}
After obtaining a search direction $(\Delta x^k, \Delta s^k, \Delta z^k)$, we need to compute the step size $\alpha$ such that the new iterate $(x^{k+1}, s^{k+1}, z^{k+1}) = (x^{k}, s^{k}, z^{k}) + \alpha (\Delta x^k, \Delta s^k, \Delta z^k)$ satisfies conic feasibility and stays close to the central path.  The latter condition is crucial for the convergence of an IPM. There are several ways of defining the proximity to the central path depending on the scaling strategy used in an IPM~\cite{Nesterov98,Dahl21,Hypatia}. 

In this paper we use the \textit{functional} proximity measure~\cite{Nesterov98,Serrano15},
\begin{align}
	\Phi(s,z) := \nu \ln(s^\top z) + f(s) + f^*(z) - \nu \ln(\nu) + \nu, \label{proximity-measure}
\end{align}
for the proximity measurement, which requires evaluation of \emph{both} the barrier $f$ and its conjugate $f^*$. The sequence $(s^k,z^k)$ generated by an IPM converges to a global optimum if $\Phi(s^k,z^k) \le 1, \forall k$. The property of $\Phi(s,z)$ is mentioned in Lemma 5.2.2 and the connection between them is discussed in Theorem 7.0.6 of~\cite{Serrano15}.

For symmetric cones, the conjugate barrier function is the same as the original one, i.e. $f = f^*$. Hence, we have analytical solutions for both $f(s^k)$ and $f^*(z^k)$ at each interior point iteration. For nonsymmetric cones, we generally will not have analytical expressions for both $f(s)$ and $f^*(z)$, but will rather have one with an analytical solution and the other computed iteratively, which amounts to computation of the conjugate gradient. 

\subsection{Computation of conjugate gradient}
For an interior point algorithm in which the dual barrier value $f^*(z)$ is available from Theorem~\ref{theorem-barrier},the corresponding primal barrier value $f(s)$ can be computed via the dual barrier function via~\eqref{fundamental-conjugate-barrier-2}, i.e.\ using 
\begin{align*}
	f(s) = - \nu - f^*(-g(s)).
\end{align*}
Hence, the main effort is to compute the primal conjugate gradient via~\eqref{fundamental-conjugate-gradient}, i.e.\ using 
\begin{align}
	g(s) = \arg \sup_{w \in \mathcal{K}^*} \{-\langle s,w \rangle - f^*(w) \}. \label{eq-conjugate-gradient}
\end{align}
We next detail how to compute the conjugate gradient of dual barrier function efficiently for generalized power cones, power mean cones and relative entropy cones.
\subsubsection{Generalized power cone}
The barrier function of $\mathcal{K}_{\text {gpow }(\alpha,d_1,d_2)}^*$ is defined in~\eqref{dual-barrier-genpow} with $\nu=d_1+1$. We denote $\varphi(u)=\prod_{i \in [d_1]} \left(u_i/\alpha_i\right)^{2 \alpha_i}$, and $\zeta(u,w)=\varphi(u) - \|w\|^2$. Our goal is to find a conjugate gradient $g_{(p,r)}$ such that $-g^*(-g_{(p,r)}) = (p,r) \in \mathcal{K}_{\text{gpow}(\alpha,d_1,d_2)}$, which is summarized as follows:
\begin{theorem}\label{theorem-cent-genpow}
	Suppose $(p,r) \in \mathcal{K}_{\mathrm{gpow}(\alpha,d_1,d_2)}$. The corresponding primal conjugate gradient $g_{(p,r)} = (g_p,g_r)$ can be obtained via
	\begin{align*}
		g_{p_i} = -\frac{(1+\alpha_i)+\alpha_i g_{\|r\|}\|r\|}{p_i}, \ \forall i \in [d_1], \qquad g_r = \frac{g_{\|r\|}\cdot r}{\|r\|},
	\end{align*}
	where $g_{\|r\|}$ is the root of
	\begin{align}
		h(x) = \sum_{i \in [d_1]}2\alpha_i \ln \left(\|r\| \cdot x + \frac{1+\alpha_i}{\alpha_i}\right) - \ln \left(\frac{2 x}{\|r\|}+x^2\right) - \sum_{i \in [d_1]}2\alpha_i \ln(p_i). \label{genpow-compute-gr-scalar}
	\end{align}
	The root of $h(x)$ is computable via the $1$-dimensional Newton-Raphson method, which converges quadratically if we choose the initial point $x_0$ as
	\begin{align*}
		x_0 = -\frac{1}{\|r\|} + \frac{\psi \|r\| + \sqrt{\chi \left(\frac{\chi}{\|r\|^2} + \psi^2- 1\right)}}{\chi - \|r\|^2},
	\end{align*}
	where $\chi = \prod_{i \in [d_1]} p_i^{2\alpha_i},\psi = 1/(\sum_{i \in [d_1]}\alpha_i^2)$.
\end{theorem}
\begin{proof}[Proof of Theorem~\ref{theorem-cent-genpow}]
	Following the preceding discussion, the dual gradient is 
	\begin{subequations}
		\begin{align}
			g^*_{u_i} &= \frac{-2\alpha_i \varphi(u)}{u_i \zeta(u,w)} - \frac{1-\alpha_i}{u_i}, \label{genpow-conjugate-grad-1}\\
			g^*_{w_i} &= \frac{2w_i}{\zeta(u,w)}.\label{genpow-conjugate-grad-2}
		\end{align}
	\end{subequations}
	If $r = 0$, we have $-\frac{2(-g_{r_i})}{\zeta(p,r)} = r_i = 0, \forall i \in [d_2]$ and then 
	\begin{align*}
		\begin{aligned}
			g_{p_i} &= -\frac{1+\alpha_i}{p_i}, \quad  \forall i \in [d_1], \\
			g_{r_i} & =0, \quad \forall i \in [d_2],
		\end{aligned}
	\end{align*}
	from $\varphi(-g_p) = \zeta(-g_p,0)$.
	
	Suppose instead $r \neq 0$.  We will denote $\zeta = \zeta(-g_p,-g_r)$ for simplicity of notation and will assume $r \ge 0$ without loss of generality. Suppose $Q$ is an orthonormal transformation that maps $r \in \mathbb{R}^{d_2}$ to a vector of zeros except one at the first entry that is equal to $\|r\|$, i.e. $[\|r\|, 0, \dots]^\top$.  Then the conjugate gradient under the transformation $Q$ satisfies $g_r = Q^\top g_{Qr}$ due to~\eqref{eq-conjugate-gradient}, where the first entry of $g_{Qr}$ is $g_{\|r\|}$ and all other entries are zeros. From~\eqref{genpow-conjugate-grad-2} and~\eqref{fundamental-bilinear-map}, have
	\begin{align}
		g_r = \frac{\zeta r}{2} \label{genpow-primal-grad-1}.
	\end{align}
	Due to the invariance of $\zeta$ under orthonormal transformations, we then obtain
	\begin{align*}
		\zeta := \zeta(-g_p,-g_r) = \zeta(-g_p,-Q^\top g_{Qr}) = \zeta(-g_p,[-g_{\|r\|}; 0; \dots]) \overset{\eqref{genpow-primal-grad-1}}{=} \frac{2g_{\|r\|}}{\|r\|},
	\end{align*}
	which can be substituted back into~\eqref{genpow-primal-grad-1} to produce
	\begin{align*}
		g_r = \frac{g_{\|r\|} \cdot r}{\|r\|}.
	\end{align*}
	From~\eqref{genpow-conjugate-grad-1} and~\eqref{fundamental-bilinear-map} we can write the elements of the subvector $p$ as
	\begin{align}
		p_i = - \left(\frac{-2\alpha_i \varphi(-g_p)}{-g_{p_i} \zeta} - \frac{1-\alpha_i}{-g_{p_i}}\right), \ \forall i \in [d_1]. \label{genpow-pi}
	\end{align}
	which can be rearranged to
	\begin{align}\label{genpow-pi_2}
		p_i(-g_{p_i})\zeta = 2 \alpha_i \varphi(-g_p) + (1-\alpha_i)\zeta.
	\end{align}
	Applying $\varphi(\cdot)$ over both sides of \eqref{genpow-pi_2} for all $i$, we obtain
	\begin{align*}
		\prod_{i \in [d_1]}\alpha_i^{2\alpha_i} \cdot \varphi(p) \varphi(-g_p)\zeta^2 = \varphi(2\alpha_i \varphi(-g_p) + (1-\alpha_i)\zeta).
	\end{align*}
	Substituting
	\begin{align}
		\varphi(-g_p)= \zeta + \|g_r\|^2 = \zeta + \|Q^\top g_{Qr}\|^2 = \zeta + g_{\|r\|}^2 = \frac{2 g_{\|r\|}}{\|r\|}+g_{\|r\|}^2 \label{genpow-varphi}
	\end{align}
     yields
	\begin{align*}
		\implies & \prod_{i \in [d_1]}\alpha_i^{2\alpha_i} \cdot \varphi(p) \left(\frac{2 g_{\|r\|}}{\|r\|}+g_{\|r\|}^2\right)\left(\frac{2 g_{\|r\|}}{\|r\|}\right)^2 = \varphi \left(2\alpha_i \left(\frac{2 g_{\|r\|}}{\|r\|}+g_{\|r\|}^2\right) + (1-\alpha_i)\frac{2 g_{\|r\|}}{\|r\|}\right)\\
		\implies & \prod_{i \in [d_1]} p_i^{2\alpha_i} \left(\frac{2 g_{\|r\|}}{\|r\|}+g_{\|r\|}^2\right) \cdot \frac{4 g^2_{\|r\|}}{\|r\|^2} = \prod_{i \in [d_1]}\left(2g_{\|r\|}^2 + \frac{(1+\alpha_i)}{\alpha_i}\frac{2 g_{\|r\|}}{\|r\|}\right)^{2\alpha_i}\\
		\implies & \prod_{i \in [d_1]} p_i^{2\alpha_i} \left(\frac{2 g_{\|r\|}}{\|r\|}+g_{\|r\|}^2\right) = \prod_{i \in [d_1]}\left(g_{\|r\|}\cdot \|r\| + \frac{1+\alpha_i}{\alpha_i}\right)^{2\alpha_i},
	\end{align*}
	where computing $g_{\|r\|}$ reduces to finding the root of the function $h(x)$ defined in~\eqref{genpow-compute-gr-scalar}. 

	We next show that $h(x)$ is convex and monotonically decreasing. The gradient and Hessian of $h$ are
	\begin{align*}
		h'(x) & = \sum_{i \in [d_1]}\frac{2 \alpha_i \|r\|}{\|r\| \cdot x + \frac{1+\alpha_i}{\alpha_i}} - \frac{2(\|r\| + \frac{1}{x})}{\|r\|\cdot x + 2} \\
		& < \sum_{i \in [d_1]}\frac{2 \alpha_i \|r\|}{\|r\| \cdot x + 2} - \frac{2(\|r\| + \frac{1}{x})}{\|r\|\cdot x + 2} = -\frac{\frac{2}{x}}{\|r\| \cdot x + 2} < 0 \\
		h''(x) & = -\sum_{i \in [d_1]}\frac{2 \alpha_i \|r\|^2}{(\|r\| \cdot x + \frac{1+\alpha_i}{\alpha_i})^2} - \frac{-\frac{2}{x^2}(\|r\|\cdot x + 2)-2(\|r\| + \frac{1}{x})\|r\|}{(\|r\|\cdot x + 2)^2} \\
		& = -\sum_{i \in [d_1]}\frac{2 \alpha_i^3}{(\alpha_i x + \frac{1+\alpha_i}{\|r\|})^2} + \frac{\frac{4\|r\|}{x} + \frac{4}{x^2} + 2\|r\|^2}{(\|r\|\cdot x + 2)^2}.
	\end{align*}
	Since $h'(x) < 0$ holds for $x > 0$ and $h(\hat{x}) > 0$, the root of $h$ is unique over $x>0$. 
	
	Consider next the function $g(\beta) = -\frac{\beta^3}{\left(\beta x + \frac{1+\beta}{\|r\|}\right)^2}$, which is concave over the interval $[0,1]$ when $x > 0$. 
	We also have $g(0) = 0$, which implies
	\begin{align*}
		g(t\beta) \ge tg(\beta)+ (1-t)g(0) = tg(\beta), \ \forall t,\beta \in [0,1].
	\end{align*}
	Therefore,
	\begin{align*}
		h''(x) & = 2\sum_{i \in [d_1]}g(\alpha_i) + \frac{\frac{4\|r\|}{x} + \frac{4}{x^2} + 2\|r\|^2}{(\|r\|\cdot x +2)^2} \ge 2\sum_{i \in [d_1]}\alpha_i g(1) + \frac{\frac{4\|r\|}{x} + \frac{4}{x^2} + 2\|r\|^2}{(\|r\|\cdot x +2)^2} \\
		& = -\frac{2}{(x+\frac{2}{\|r\|})^2} + \frac{\frac{4\|r\|}{x} + \frac{4}{x^2} + 2\|r\|^2}{(\|r\|\cdot x +2)^2} = \frac{\frac{4\|r\|}{x} + \frac{4}{x^2}}{(\|r\|\cdot x +2)^2} > 0
	\end{align*}
	shows $h$ is convex over $x>0$. Since $h$ is convex and decreasing, the root of $h(x)$ is unique and a root finding Newton-Raphson method converges quadratically if it starts with $0< x_0 \le \hat{x}$ where $h(\hat{x}) > 0$~\cite[Thm.\ 1.9]{Suli2003}. Since $q(t) = \ln \left(\|r\| \cdot x + \frac{1}{t} +1 \right)$ is convex when $t \in [0,1], x \ge 0$, we can apply the Jensen's inequality and obtain
	\begin{align*}
		h(x) &\ge 2 \ln(\|r\|\cdot x+\frac{1}{\sum_{i \in [d_1]}\alpha_i^2}+1) - \ln \left(\frac{2 x}{\|r\|}+x^2\right) - \sum_{i \in [d_1]}2\alpha_i \ln(p_i), 
	\end{align*}
	where a positive root of the right part above is 
	\begin{align*}
		\hat{x} = -\frac{1}{\|r\|} + \frac{\psi \|r\| + \sqrt{\chi \left(\frac{\chi}{\|r\|^2} + \psi^2- 1\right)}}{\chi - \|r\|^2}
	\end{align*}
	with $\chi = \prod_{i \in [d_1]} p_i^{2\alpha_i}, \psi = 1/(\sum_{i \in [d_1]}\alpha_i^2)$, which implies $h(\hat{x}) \ge 0$ and we can set it as the starting point of the Newton-Raphson method.
	After obtaining $g_{\|r\|}$ from~\eqref{genpow-compute-gr-scalar}, we get $\zeta = \frac{2g_{\|r\|}}{\|r\|}$ and the primal conjugate gradient
	\begin{align*}
			g_r &= \frac{g_{\|r\|}\cdot r}{\|r\|}, \quad g_{p_i} = -\frac{1}{p_i \zeta}\left((1-\alpha_i)\zeta + 2\alpha_i(\zeta + g_{\|r\|}^2)\right) = -\frac{1}{p_i}(1+\alpha_i+\alpha_i g_{\|r\|}\|r\|),
	\end{align*}
	via~\eqref{genpow-pi} and~\eqref{genpow-varphi}.
\end{proof}

\subsubsection{Power mean cone}
For the dual barrier function $f_{\mathrm{powm}}^*(u,w)$ defined in~\eqref{dual-barrier-powm}, we denote $\varphi(u)=\prod_{i \in [d]} \left(\frac{u_i}{\alpha_i}\right)^{\alpha_i}$, $\zeta(u,w)=\varphi(u)+ w$ and $\tau_i=\frac{\alpha_i}{u_i \zeta}, \forall i \in [d]$. The computation of the primal conjugate gradient $g_{(p,r)}$ is summarized in the theorem below.
\begin{theorem}\label{theorem-cent-powm}
	Suppose $(p,r) \in \mathcal{K}_{\mathrm{powm}(\alpha,d)}$. The corresponding primal conjugate gradient $g_{(p,r)}=(g_p,g_r)$ can be obtained via
	\begin{align*}
		g_{p_i}=-\frac{1}{p_i}-\frac{\alpha_i}{p_i}\left(1+ r g_r\right),
	\end{align*}
	where $g_r$ can be computed numerically depending on the sign of $r$:
	\begin{enumerate}
		\item $\mathbf{r > 0}$: We can obtain $g_{r} = 1/\hat{x}$, where $\hat{x}$ is the root of
		\begin{align*}
			h(x) = \sum_{i \in [d]} \alpha_i \ln \left(\frac{1+\alpha_i}{\alpha_i p_i}\cdot x+\frac{r}{p_i}\right) - \ln \left(1+ \frac{1}{r+x}\right).
		\end{align*}
		The root of $h(x)$ can be computed via the $1$-dimensional Newton-Raphson method, which converges quadratically if we start from $x_0 = 0$.
		\item $\mathbf{r < 0}$: We can obtain $g_r = \left(\varphi-\sqrt{\varphi^2 + 4/r^2}\right)/2 - 1/r$, where $\varphi = 1/\hat{x}$ and $\hat{x}$ is the root of
		\begin{align*}
			h(x) = \sum_{i \in [d]} \alpha_i \ln \left(\frac{x}{\alpha_i p_i}+\frac{r + \sqrt{r^2+4x^2}}{2p_i}\right). 
		\end{align*}
		We can compute the root of $h(x)$ via the $1$-dimensional Newton-Raphson method, which converges quadratically if we start from $x_0 =  \left(\frac{\alpha_i p_i}{1+\alpha_i}\right)^{\alpha_i}$.
		\item $\mathbf{r = 0}$: We have
		\begin{align*}
			g_{p_i} = -\frac{1}{p_i} - \frac{\alpha_i}{p_i}, \qquad g_r = \varphi(-g_p)/2.
		\end{align*}
	\end{enumerate}
\end{theorem}
\begin{proof}[Proof of Theorem~\ref{theorem-cent-powm}]
	The dual gradient of~\eqref{dual-barrier-powm} is
	\begin{align*}
		\begin{aligned}
			&g_u^*=-\tau_i \varphi-\frac{1-\alpha_i}{u_i} \\
			& g_w^*=-\frac{1}{\zeta}-\frac{1}{w}.	
		\end{aligned}
	\end{align*}
	Our goal is to find $g_{(p, r)}$ such that $-g^*\left(-g_{(p, r)}\right)=\left(\begin{array}{l}p \\ r\end{array}\right) \in \mathcal{K}_{\text {powm}}$, which is
	\begin{align*}
		\begin{aligned}
			&p = -\frac{\alpha_i \varphi(-g_{(p,r)})}{g_{p_i}\zeta(-g_{(p,r)})} - \frac{1-\alpha_i}{g_{p_i}} = -\frac{1}{g_{p_i}}- \frac{\alpha_i g_r}{g_{p_i}\zeta(-g_{(p,r)})},\\
			&r=\frac{1}{\zeta(-g_{(p,r)})} - \frac{1}{g_r}
		\end{aligned}
 	\end{align*}
	with $g_r > 0$.  We then obtain the primal gradient $g_{(p,r)}$ as
	\begin{align}
		\begin{aligned}
			&g_{p_i}=-\frac{1}{p_i}-\frac{\alpha_i}{p_i}\left(1+ r g_r\right),\\ &\zeta(-g_{(p,r)})=\frac{1}{r+\frac{1}{g_r}}=\frac{g_r}{1+r g_r} > 0.
		\end{aligned}\label{power-mean-primal-grad}
	\end{align}
	Due to the definition of $\zeta(\cdot)$ and the two equalities above, we obtain
	\begin{align}
		&\frac{g_r}{1+ rg_r}=\prod_{i \in[d]}\left(\frac{1}{\alpha_i p_i}+\frac{1}{p_i}\left(1+ rg_r \right)\right)^{\alpha_i}- g_r. \label{root-eq-powm}
	\end{align}
We compute $g_p,g_r$ in three different ways based on the value of $r$:
	\begin{enumerate}
		\item When $r > 0$, ~\eqref{root-eq-powm} reduces to finding the root $\hat{x}$ of $h(x)$ if we regard $1/g_r$ as $x$,
		\begin{align*}
			\begin{aligned}
				&h(x):=\sum_{i \in [d]} \alpha_i \ln \left(\frac{1+\alpha_i}{\alpha_i p_i}\cdot x+\frac{r}{p_i}\right) - \ln \left(1+ \frac{1}{r+x}\right).
			\end{aligned}
		\end{align*}	
		The gradient and the hessian of $h(x)$ are
		\begin{align*}
			\begin{aligned}
				h^{\prime}(x) &= \sum_{i \in [d]} \frac{\alpha_i}{x + \frac{\alpha_i r}{1+\alpha_i}} + \frac{1}{(r+x)(r+x+1)} > 0,\\
				h^{\prime \prime}(x) &= -\sum_{i \in [d]} \frac{\alpha_i}{(x + \frac{\alpha_i r}{1+\alpha_i})^2} - \frac{1}{(r+x)^2} + \frac{1}{(r+x+1)^2} < 0,
			\end{aligned}
		\end{align*}
		when $x \ge 0$, which implies $h(x)$ is concave and increasing. Moreover, we have
		\begin{align*}
			h(0) = \sum_{i \in [d]} \alpha_i \ln \left(\frac{r}{p_i}\right) - \ln \left(1+ \frac{1}{r}\right) = \ln \left(\frac{r^2}{(1+r)\prod_{i=1}^d p_i^{\alpha_i}}\right) \le \left(\frac{r^2}{(1+r)r}\right) < 0.
		\end{align*}
		Hence, the root of $h(x)$ is unique and a root finding Newton-Raphson method converges quadratically when it starts with $x_0 = 0$ (Theorem 1.9 in~\cite{Suli2003}). Finally, we obtain the primal gradient $g_{(p,r)}$ via~\eqref{power-mean-primal-grad}, given $g_r = 1/\hat{x}$.
		\item When $r < 0$, we design a different strategy to compute the conjugate gradient. Due to the definition of $\varphi(\cdot)$ and~\eqref{power-mean-primal-grad}, we have 
		\begin{align*}
			\varphi(-g_p) = \frac{g_r}{1+rg_r}+ g_r,
		\end{align*}
		which (abbreviating $\varphi(-g_p)$ as $\varphi$) can be written as
		\begin{align}
			g_r = \frac{\varphi-\sqrt{\varphi^2 + \frac{4}{r^2}}}{2} - \frac{1}{r} \label{power-mean-gr}
		\end{align}
		given $g_r \ge -1/r$. Substituting~\eqref{power-mean-gr} into~\eqref{root-eq-powm} yields
		\begin{align*}
			\varphi = \prod_{i \in[d]}\left(\frac{1}{\alpha_i p_i}+\frac{r\varphi + \sqrt{r^2\varphi^2+4}}{2p_i}\right)^{\alpha_i},
		\end{align*}
	which is equivalent to finding the root of $h(x)$ if we regard $1/\varphi$ as $x$, where
	\begin{align}
		h(x) = \sum_{i \in [d]} \alpha_i \ln \left(\frac{x}{\alpha_i p_i}+\frac{r + \sqrt{r^2+4x^2}}{2p_i}\right). \label{h-def-power-mean}
	\end{align}
	We can obtain
	\begin{align*}
		&h'(x) = \sum_{i \in [d]}\alpha_i \cdot \frac{1+\frac{2\alpha_i x}{\sqrt{r^2+4x^2}}}{x+\frac{\alpha_i}{2}(r+\sqrt{r^2+4x^2})} > 0
	\end{align*}
	and $h''(x) < 0$ (See Appendix~\ref{appendix-power-mean}) when $x > 0$. Moreover,
	\begin{align*}
		h(x_0) & = \sum_{i \in [d]} \alpha_i \ln \left(\frac{x_0}{\alpha_i p_i}+\frac{2x_0^2}{p_i(\sqrt{r^2+4x_0^2}-r)}\right) < \sum_{i \in [d]} \alpha_i \ln \left(\frac{x_0}{\alpha_i p_i}+\frac{x_0}{p_i}\right) \\
		& \le \ln(x_0) + \sum_{i \in [d]} \alpha_i \ln \left(\frac{1+\alpha_i}{\alpha_i p_i}\right) = 0.
	\end{align*}
	Hence, the root $\hat{x}$ of $h(x)$ is unique and a root finding Newton-Raphson method converges quadratically when it starts with $x_0 = \left(\frac{\alpha_i p_i}{1+\alpha_i}\right)^{\alpha_i}$ with $h(x_0) < 0$. Given $\varphi = 1/\hat{x}$, we obtain the primal gradient $g_r$ via~\eqref{power-mean-gr} and $g_p$~\eqref{power-mean-primal-grad}.
	\item When $r = 0$,~\eqref{power-mean-primal-grad} becomes
	\begin{align*}
		g_{p_i} = -\frac{1}{p_i} - \frac{\alpha_i}{p_i}, \qquad \zeta(-g_{(p,r)}) = g_r.
	\end{align*}
	Then, $\varphi(-g_p) -g_r = \zeta(-g_{(p,r)}) = g_r$ due to the definition of $\varphi(\cdot)$, which yields $g_r = \varphi(-g_p)/2$.
	\end{enumerate}
	
\end{proof}

\subsubsection{Relative entropy cone}
For the barrier function $f_{\mathrm{rel}}^*(u,v,w)$ defined in~\eqref{dual-barrier-rel}, the primal conjugate gradient can be obtained from the following theorem.

\begin{theorem}\label{theorem-cent-rel-entropy}
	Suppose we have a point $(p,q,r) \in \mathcal{K}_{\text {rel}}$ where $p \in \mathbb{R}, q,r \in \mathbb{R}^d$. The corresponding primal conjugate gradient $g_{(p,q,r)}=(g_p,g_q,g_r)$ can be obtained via
	\begin{align*}
		g_{q_i} = \frac{g_p r_i - 1}{q_i}, \qquad g_{r_i} = g_p\ln\left(\frac{g_p}{g_{q_i}}\right) - g_p - \frac{1}{r_i}, \ \forall i \in [d],
	\end{align*}
	where $g_p = -1/\hat{x}$. $\hat{x}$ is the root of
	\begin{align*}
		h(x): = d \cdot x + \sum_{i \in [d]} r_i \ln \left(\frac{r_i}{q_i} + \frac{x}{q_i}\right) - p,
	\end{align*}
	which can be solved via the $1$-dimensional Newton-Raphson method, which converges quadratically if we choose the initial point $x_0=0$.
\end{theorem}
\begin{proof}[Proof of Theorem~\ref{theorem-cent-rel-entropy}]
	We aim to find the conjugate gradient $g_{(p, q, r)}$ such that $-g^*\left(-g_{(p, q, r)}\right)=(p,q,r) \in \mathcal{K}_{\text {rel}}$, which becomes
	\begin{align*}
		\begin{aligned}
			p &= -\sum\nolimits_{i \in [d]} \frac{\ln\left(\frac{g_p}{g_{q_i}}\right)}{\gamma_i} - \frac{d}{g_p}, \\
			q_i &= \frac{g_p}{\gamma_i g_{q_i}} - \frac{1}{g_{q_i}}, \\
			r_i &= \frac{1}{\gamma_i},
		\end{aligned}
	\end{align*}
	where $\gamma_i$ is the abbreviation of $\gamma_i(-g_{(p,q,r)})$ defined in Section~\ref{subsection:relative-entropy-cone}. We can then obtain
	\begin{align}
		-g_{q_i} = \frac{-g_p r_i + 1}{q_i}, \qquad \gamma_i = \frac{1}{r_i}. \label{entropy-primal-gradient}
	\end{align}
	with
	\begin{align*}
		p &= -\frac{d}{g_p}-\sum_{i \in [d]} r_i \ln \left(\frac{g_p q_i}{g_p r_i - 1}\right)\\
		& = -\frac{d}{g_p} + \sum_{i \in [d]} r_i \ln \left(\frac{r_i}{q_i} - \frac{1}{g_p q_i}\right),
	\end{align*}
	which is a nonlinear equation w.r.t.\ $g_p$. Suppose we define $h(x)$ as
	\begin{align*}
		h(x): = d \cdot x + \sum_{i \in [d]} r_i \ln \left(\frac{r_i}{q_i} + \frac{x}{q_i}\right) - p.
	\end{align*}
	We can compute the root $\hat{x}$ of $h(x)$ and $g_p$ is indeed $-\hat{x}^{-1}$. Since $h(x)$ is concave and increasing, the root of $g(x)$ is unique and a root finding Newton-Raphson method converges quadratically when it starts with $0 \le x_0 \le \hat{x}$ where $g(x_0) < 0$ (Theorem 1.9 in~\cite{Suli2003}). A feasible choice is $x_0 = 0$ as $h(0) = \sum_{i \in [d]}r_i\ln \left(r_i/q_i\right) < p$ due to $(p,q,r) \in \mathcal{K}_{\text {rel}}$. Finally, we can obtain $g_{q_i}, g_{r_i}, \forall i \in [d]$ via~\eqref{entropy-primal-gradient} and
	\begin{align*}
		g_{r_i} = g_p\ln\left(\frac{g_p}{g_{q_i}}\right) - g_p - \frac{1}{r_i}
	\end{align*}
	due to the definition of $\gamma_i$.
\end{proof}

\section{A dual-scaling IPM for nonsymmetric conic optimization}\label{section-implementation}
We summarize the implementation of our dual-scaling IPM algorithm in this section.
\subsection{Initialization of $s^0,z^0$}
The initialization of symmetric cones and three-dimensional nonsymmetric cones, i.e. power and exponential cones, follows the Mosek setting as discussed in~\cite{Dahl21}, where $x^0 = 0, \kappa^0= \tau^0 = 1$ and $s^0,z^0$ are set to 
\begin{align}
	s^0 = z^0 = -g^*(z^0), \label{unit-initialization}
\end{align}
which are on the central path with $\mu^0 = 1$. This strategy holds for $\mathcal{K}_{\mathrm{gpow}(\alpha,d_1,d_2)}$, and we initialize $s^0,z^0$ by 
\[
s^0=z^0=\left(\left(\sqrt{1+\alpha_i}\right)_{i \in [d_1]}, 0_{d_2}\right).
\] 

For $\mathcal{K}_{\mathrm{powm}(\alpha,d)},\mathcal{K}_{\mathrm{rel}}$, we can not find $s^0,z^0$ satisfying~\eqref{unit-initialization}, so we initialize $s^0,z^0$ as in Hypatia~\cite{Hypatia}, where $s^0, z^0$ satisfy $s^0 = -g^*(z^0)$. Note that $\mu^0 = 1$ still holds in this setting.
\subsection{Algorithmic sketch}
Our dual-scaling IPM proceeds as follows at each iteration $k$:
\begin{enumerate}
	\item \textit{Update residuals, gap and check termination or infeasibility:} We compute residuals by
	$$
	\begin{array}{ll}
		r_x = -G^\top y^k-A^\top z^k-c \tau^k, & r_y = G x^k-h \tau^k, \\
		r_\tau = \kappa^k+c^\top x^k +h^\top y^k +b^\top z^k, & r_z = s^k+A x^k- b \tau^k,
	\end{array}
	$$
	and $\mu^k = \frac{{s^k}^\top z^k +\kappa^k \tau^k}{\nu + 1}$. The computation of duality gap, the termination check and the infeasibility detection follows ECOS~\cite{Domahidi13}.
	\item \textit{Affine step:} The affine direction (predictor) is the solution of
	\begin{align*}
		\begin{aligned}
			& \left[\begin{array}{cccc}
				0 & G^\top & A^\top & c \\
				-G & 0 & 0 & h \\
				-A & 0 & 0 & b \\
				-c^\top & -h^\top & -b^\top & 0
			\end{array}\right]\left[\begin{array}{l}
				\Delta x_a \\
				\Delta y_a \\
				\Delta z_a \\
				\Delta {\tau}_a
			\end{array}\right]  - \left[\begin{array}{c}
				0 \\
				0 \\
				\Delta s_a \\
				\Delta {\kappa}_a
			\end{array}\right] = \left[\begin{array}{l}
				r_x \\
				r_y \\
				r_z \\
				r_\tau
			\end{array}\right]\\
			& \qquad \mu^k H^*(z^k)\Delta z_a + \Delta s_a = -s^k, \quad \tau^k \Delta \kappa_a + \kappa^k \Delta \tau_a = - \tau^k \kappa^k,
		\end{aligned}
	\end{align*}
	which tries to remove residuals of the linearized model at iteration $k$. We can compute the maximal step size $\alpha_{a}$ such that $(s^k+\alpha_{a} \Delta s_a, z^k+\alpha_{a} \Delta z_a, \tau^k+\alpha_{a} \Delta \tau_a, \kappa^k +\alpha_{a} \Delta \kappa_a)$ resides in $\mathcal{F}$.
	\item \textit{Update the sparse decomposition of $H^*(z)$:} We update the augmented sparse decomposition of the dual Hessian $H^*(z)$ for nonsymmetric cones as discussed in Section~\ref{section:augmented-sparsity} and Section~\ref{section-sparse-decomposition}.
	\item \textit{Combined step:} The weight of centrality $\sigma$ is set to $(1-\alpha_a)^3$ empirically and then used for the computation of centering direction (corrector) via
	\begin{align*}
		\begin{aligned}
			& \left[\begin{array}{cccc}
				0 & G^\top & A^\top & c \\
				-G & 0 & 0 & h \\
				-A & 0 & 0 & b \\
				-c^\top & -h^\top & -b^\top & 0
			\end{array}\right]\left[\begin{array}{l}
				\Delta x_c \\
				\Delta y_c \\
				\Delta z_c \\
				\Delta {\tau_c}
			\end{array}\right]  - \left[\begin{array}{c}
				0 \\
				0 \\
				\Delta s_c \\
				\Delta {\kappa_c}
			\end{array}\right] = (1-\sigma)\left[\begin{array}{l}
				r_x \\
				r_y \\
				r_z \\
				r_\tau
			\end{array}\right]\\
			& \qquad \mu^k H^*(z^k)\Delta z_c + \Delta s_c = -s^k - \sigma \mu^k g^*(z^k) - \eta(\Delta s_a, \Delta z_a), \\
			& \qquad \quad \tau^k \Delta \kappa_c + \kappa^k \Delta \tau_c = - \tau^k \kappa^k + \sigma \mu^k - \Delta \tau_a \Delta \kappa_a,
		\end{aligned}
	\end{align*}
	where $\Delta \tau_a \Delta \kappa_a$ and the function $\eta(\Delta s_a, \Delta z_a)$ are higher-order corrections given information from the affine directions. We only use Mehrotra's correction~\cite{Mehrotra92} for symmetric cones. Likewise, we compute the largest step size $\alpha_c$ ensuring $(s^k+\alpha_{c} \Delta s_c, z^k+\alpha_{c} \Delta z_c, \tau^k+\alpha_{c} \Delta \tau_c, \kappa^k +\alpha_{c} \Delta \kappa_c)$ stays inside $\mathcal{F}$ and the proximity measure satisfies $\Phi(s^k+\alpha_{c} \Delta s_c, z^k+\alpha_{c} \Delta z_c) \le 1$ as discussed in Section~\ref{section-centrality-check}. 
	\item \textit{Update iterates:} At the end of each iteration $k$, we obtain the new iterate $(x^{k+1},y^{k+1},s^{k+1},z^{k+1},\kappa^{k+1},\tau^{k+1})$ by
	\begin{align*}
	\resizebox{.9 \textwidth}{!}
	{$(x^{k+1},y^{k+1},s^{k+1},z^{k+1},\kappa^{k+1},\tau^{k+1}) := (x^k,y^k,s^k,z^k,\kappa^k,\tau^k) + \alpha_c (\Delta x_c,\Delta y_c, \Delta s_c, \Delta z_c,\Delta \kappa_c,\Delta \tau_c).$}
	\end{align*}
\end{enumerate}

\section{Experiments}\label{section-experiments}
In our experiments, we choose two examples from Hypatia~\cite{Coey23} to show the effectiveness of our sparse implementation for generalized power cones, the discrete maximum likelihood and the maximum volume hypercube problems. We test these examples with 4 different solver configurations: 
\begin{enumerate}
	\item \texttt{Clarabel-GenPow} is the proposed sparse implementation of generalized power cones in Clarabel; 
	\item \texttt{Clarabel-Pow} uses Clarabel but transforms the generalized power cone to a product of three-dimensional power cones as formulated in~\cite{Chares09,mosek};
	\item \texttt{Hypatia} uses the Hypatia solver, which supports the generalized power cone constraints directly;
	\item \texttt{Mosek} denotes the use of Mosek with transforming the generalized power cone to a product of three-dimensional power cones.
\end{enumerate}
The convergence tolerance is set to $\epsilon=10^{-8}$ for all solvers in all tests. We utilize the dual scaling strategy (same as ECOS) for Clarabel in our test for fair comparison between the direct support of generalized power cones (configuration (1)) and the transformation to a product of three-dimensional power cones (configuration (2)). 

We also test the sparse implementation of the power mean cone, called \texttt{Clarabel-PowM}, on these same two examples, since it produces the same results as the generalized power cone for both cases.

In addition, we choose the entropy maximization problem from Mosek's Cookbook~\cite{mosek} for testing relative entropy cones. We test it with 4 different solver settings: \begin{enumerate}
	\item \texttt{Clarabel-Entropy} is the proposed sparse implementation of relative entropy cones via Clarabel; 
	\item \texttt{Clarabel-Exp} uses Clarabel but transforms the relative entropy cone to a product of three-dimensional exponential cones; 
	\item \texttt{ECOS} uses the ECOS solver, which supports three-dimensional exponential cones; 
	\item \texttt{Mosek} denotes the use of Mosek solver with three-dimensional exponential cones.
\end{enumerate} The termination accuracy is set to be $\epsilon=10^{-8}$ and both setting (1) and (2) use the dual scaling strategy for this test.
\subsection{Discrete maximum likelihood}
\subsubsection{Test of generalized power cones}
The discrete maximum likelihood problem has the form
\begin{align}
\begin{aligned}
	& \max \quad t \\
	\mathrm{s.t.} \quad & \sum_{i \in [n]} x_i = 1, \\
	&  (x,t) \in \mathcal{K}_{\text {gpow }(\alpha,n,1)},
\end{aligned}\label{example-1}
\end{align}
and we test \texttt{Clarabel-GenPow} and \texttt{Hypatia} on it directly. Following~\cite{mosek}, we can rewrite~\eqref{example-1} by replacing the generalized power cone constraint with the products of three-dimensional power cones where $p_i = \alpha_i/(\sum_{ 1\le j \le i}\alpha_j)$. This is the form used by \texttt{Clarabel-Pow} and \texttt{Mosek},
\begin{align}
	\begin{aligned}
		& \max \quad t \\
		\mathrm{s.t.} \quad & \sum_{i \in [n]} x_i = 1, x \ge 0,\\
		& x_1^{1 -p_2} x_2^{p_2} \geq\left\vert z_3\right\vert, \\
		& z_3^{1-p_3} x_3^{p_3} \geq\left\vert z_4\right\vert, \\
		& \cdots \\
		& z_{n-1}^{1-p_{n-1}} x_{n-1}^{p_{n-1}} \geq\left\vert z_n\right\vert, \\
		& z_n^{1-p_n} x_n^{p_n} \geq \vert t\vert.
	\end{aligned}\label{example-1-extended}
\end{align}

 Computational results are given in Table~\ref{table-1-1} and Table~\ref{table-1-2} for results of  dimension $n$ from $100$ to $2500$, which record the total computational time and the number of  iterations for different dimensions $n$ of the generalized power cone respectively.
\ifpreprint
\begin{table}[H]
	\begin{center}
		\begin{tabular}{ |c||c|c|c|c|  }
			\hline
			n & \texttt{Clarabel-GenPow} & \texttt{Clarabel-Pow} & \texttt{Mosek} & \texttt{Hypatia}\\
			\hline
			100   & 3.96 & 7.67 & 16 & 33 \\
			\hline
			500   & 6.42 & 23.9 & 31 & 334 \\
			\hline
			2500  & 23.7 & 142 & 125 & 22190 \\
			\hline
		\end{tabular}
	\end{center}
	\caption{Time (ms) w.r.t. dimensions (generalized power cones)\label{table-1-1}}
\end{table}
\begin{table}[H]
	\begin{center}
		\begin{tabular}{ |c||c|c|c|c|  }
			\hline
			n & \texttt{Clarabel-GenPow} & \texttt{Clarabel-Pow} & \texttt{Mosek} & \texttt{Hypatia}\\
			\hline
			100   & 14 & 19 & 11 & 22 \\
			\hline
			500   & 13 & 20 & 14 & 37 \\
			\hline
			2500  & 16 & 24 & 16 & 67 \\
			\hline
		\end{tabular}
	\end{center}
	\caption{Iteration number w.r.t. dimensions (generalized power cones)\label{table-1-2}}
\end{table}
\fi
Table~\ref{table-1-1} shows that our sparse LDL implementation for generalized power cones is very  efficient relative to the other methods. It can solve the problem within a few tens of milliseconds even for dimensions up to $n=2500$. In comparison, \texttt{Clarabel-Pow} and \texttt{Mosek} take more than $100$ ms to solve it, which is about 4-5 times slower than the proposed sparse implementation in Section~\ref{section-sparse-decomposition}. This can be explained by the transformation of a generalized power cone to a product of three-dimensional cones. Firstly, the transformation increases the total number of variables in~\eqref{example-1-extended}, which results in more non-zero entries in LDL factorization at each iteration. Secondly, it takes less time to compute the conjugate gradient for a generalized power cone than the computation for an equivalent collection of three-dimensional cones, and hence less time for centrality check of a generalized power cone in~\eqref{example-1} than that in~\eqref{example-1-extended}. Thirdly, Table~\ref{table-1-2} says \texttt{Clarabel-GenPow} takes fewer iteration numbers than \texttt{Clarabel-Pow} regardless of the choice of dimension $n$, which is consistent with the link between the self-concordant parameter $\nu$ and the complexity of the total iteration number in theory, i.e. $\mathcal{O}(\sqrt{\nu}\ln(1/\epsilon))$. 

\subsubsection{Test of power mean cones}
For the same problem as in the previous section we can replace the generalized power cone in~\eqref{example-1} with a power mean cone, 
\begin{align*}
	\begin{aligned}
		& \max \quad t \\
		\mathrm{s.t.} \quad & \sum_{i \in [n]} x_i = 1, \\
		&  (x,t) \in \mathcal{K}_{\mathrm{powm}(\alpha,n)},
	\end{aligned}
\end{align*}
which should yield the same solution as~\eqref{example-1} since the optimum of~\eqref{example-1} must be $t \ge 0$. Since \texttt{Mosek} doesn't support power mean cones directly we compare our spare implementation only with Hypatia. We vary the dimension $n$ as for the test of generalized power cones. Table~\ref{table-1-3} shows that our sparse implementation for power mean cones is significantly faster than Hypatia which is based on the Cholesky decomposition on the reduced normal system, both in total time and the averaged time per iteration. Although the results are slower than our sparse implementation of generalized power cones, it still performs better than other solvers on the discrete maximum likelihood problem. This demonstrates the efficacy of our sparse augmentation approach. Compared to \texttt{Hypatia}, \texttt{Clarabel-PowM} also scales well when the dimension $n$ becomes large.

\begin{table}[!htb]
	\begin{subtable}{.5\linewidth}
		\centering
		\begin{tabular}{ |c||c|c|  }
			\hline
			n & \texttt{Clarabel-PowM} & \texttt{Hypatia}\\
			\hline
			100   & 4.4 & 23 \\
			\hline
			500   & 10.8 & 320\\
			\hline
			2500  & 49.1 & 20926 \\
			\hline
		\end{tabular}
		\caption{Time (ms) w.r.t. dimensions}
	\end{subtable}%
	\begin{subtable}{.5\linewidth}
		\centering
		\begin{tabular}{ |c||c|c|  }
			\hline
			n & \texttt{Clarabel-PowM} & \texttt{Hypatia}\\
			\hline
			100   & 13 & 17 \\
			\hline
			500   & 14 & 33\\
			\hline
			2500  & 16 & 66 \\
			\hline
		\end{tabular}
		\caption{Iteration number w.r.t. dimensions}
	\end{subtable} 
	\bigskip
	\caption{The maximum likelihood problem via power mean cones} \label{table-1-3}
\end{table}

\subsection{Maximum volume hypercube}
\subsubsection{Test of generalized power cones}
The problem of finding the maximum volume is given as follows,
\begin{align}
\begin{aligned}
	& \max \quad t \\
	\text{s.t.} \quad & (x,t) \in \mathcal{K}_{\text {gpow }(\alpha,n,1)}, \alpha = \left[\frac{1}{n}, \dots, \frac{1}{n}\right], \\
	& Ax \in \mathcal{B}_1(\gamma), \  Ax \in \mathcal{B}_{\infty}(\gamma), 
\end{aligned}\label{example-2}
\end{align}
where $\mathcal{B}_1(\gamma) := \{x \ \mid \ \|x\|_1 \le \gamma\}$ is a $1$-norm ball, $\mathcal{B}_{\infty}(\gamma) := \{x \ \mid \ \|x\|_\infty \le \gamma\}$ is an $\infty$-norm ball and $\gamma$ is a given constant. Both norm constraints can be transformed to equivalent linear constraints and we use the native transformation from MathOptInterface.jl~\cite{MOI} in our implementation.

We set $A$ to the identity matrix in order to reduce the effect of constraints other than the generalized power cone. The solver settings are kept the same as in the previous example, along with the same transformation of the generalized power cone to a product of three-dimensional power cones as in~\eqref{example-1-extended}. We again vary the dimension $n$ from $100$ to $2500$ and present the computational results in Table~\ref{table-2-1} and Table~\ref{table-2-2}.
\ifpreprint
\begin{table}[H]
	\begin{center}
		\begin{tabular}{ |c||c|c|c|c|  }
			\hline
			n & \texttt{Clarabel-GenPow} & \texttt{Clarabel-Pow} & \texttt{Mosek} & \texttt{Hypatia}\\
			\hline
			100   & 6.41 & 9.79 & $<$10 & 26 \\
			\hline
			500   & 16.8 & 39.4 & 32 & 436 \\
			\hline
			2500  & 90.6 & 256 & 141 & 23413 \\
			\hline
		\end{tabular}
	\end{center}
	\caption{Time (ms) w.r.t. dimensions (generalized power cones)\label{table-2-1}}
\end{table}
\begin{table}[H]
	\begin{center}
		\begin{tabular}{ |c||c|c|c|c|  }
			\hline
			n & \texttt{Clarabel-GenPow} & \texttt{Clarabel-Pow} & \texttt{Mosek} & \texttt{Hypatia}\\
			\hline
			100   & 18 & 22 & 9 & 11 \\
			\hline
			500   & 18 & 27 & 11 & 16 \\
			\hline
			2500  & 18 & 29 & 13 & 20 \\
			\hline
		\end{tabular}
	\end{center}
	\caption{Iteration number w.r.t. dimensions (generalized power cones)\label{table-2-2}}
\end{table}
\fi

Once again, Table~\ref{table-2-1} shows the superior performance of \texttt{Clarabel-GenPow} over \texttt{Clarabel-Pow}, \texttt{Mosek} and \texttt{Hypatia}. Our sparse decomposition can solve~\eqref{example-2} within $100$ ms even up to dimension $n=2500$. Though the speedup seems to be less for the maximum volume hypercube problem than the discrete maximum likelihood example, it mainly comes from increased computational time due to the introduction of norm constraints rather than the efficiency of our sparse decomposition for generalized power cones. Meanwhile, we also list numerical results for both~\eqref{example-1} and~\eqref{example-2} in \texttt{Hypatia}, but it seems to take much more time than other three solvers under any setting. This is due to the fact that \texttt{Hypatia} uses the Cholesky factorization of the reduced system that factorizes a $(n+1) \times (n+1)$ dense positive semidefinite matrix for each iteration and doesn't utilize the sparsity of problem~\eqref{example-1} and~\eqref{example-2}. It was verified by our additional experiments that we tried to increase the density of $A$ in~\eqref{example-2} given a fixed dimension $n$ and found increased computational time for all settings except \texttt{Hypatia}. 

\subsubsection{Test of power mean cones}
We can again replace the generalized power cone in~\eqref{example-2} with a power mean cone, 
\begin{align*}
	\begin{aligned}
		& \max \quad t \\
		\text{s.t.} \quad & (x,t) \in \mathcal{K}_{\mathrm{powm}(\alpha,n)}, \alpha = \left[\frac{1}{n}, \dots, \frac{1}{n}\right], \\
		& Ax \in \mathcal{B}_1(\gamma), \  Ax \in \mathcal{B}_{\infty}(\gamma), 
	\end{aligned}
\end{align*}
which should yields the same solution as~\eqref{example-2} since the optimum of~\eqref{example-2} must be $t \ge 0$. We test it with $n=100,500,2500$ and Table~\ref{table-2-3} shows that the proposed implementation of power mean cones also performs much better than \texttt{Hypatia}. \texttt{Clarabel-PowM} also performs slightly slower than our sparse implementation of generalized power cones but better than \texttt{Mosek} in solving the maximum volume hypercube problem.
\begin{table}[!htb]
	\begin{subtable}{.5\linewidth}
		\centering
		\begin{tabular}{ |c||c|c|  }
			\hline
			n & \texttt{Clarabel-PowM} & \texttt{Hypatia}\\
			\hline
			100   & 6.1 & 27.0 \\
			\hline
			500   & 20.6 & 370 \\
			\hline
			2500  & 120.4 & 23272 \\
			\hline
		\end{tabular}
		\caption{Time (ms) w.r.t. dimensions}
	\end{subtable}%
	\begin{subtable}{.5\linewidth}
		\centering
		\begin{tabular}{ |c||c|c|  }
			\hline
			n & \texttt{Clarabel-PowM} & \texttt{Hypatia}\\
			\hline
			100   & 15 & 15 \\
			\hline
			500   & 18 & 17\\
			\hline
			2500  & 20 & 21 \\
			\hline
		\end{tabular}
		\caption{Iteration number w.r.t. dimensions}
	\end{subtable} 
	\caption{The maximum volume problem via power mean cones}\label{table-2-3}
\end{table}

\subsection{Entropy maximization}
The entropy maximization problem can be reformulated to the problem as follows:
\begin{align}
\begin{aligned}
	\min \quad & t \\
	\text{s.t.} \quad & \sum_{i \in [d]} p_i = 1, \ p_i \ge 0, \\
	& (r_i,q_i,p_i) \in \mathcal{K}_{\text{exp}}, \forall i \in [d]\\
	& -\sum_{i \in [d]} r_i \le t, \\
	&  \quad p \in \mathcal{P}, q \in \mathcal{Q},
\end{aligned}\label{example-3}
\end{align}
where $\mathcal{P}$ defines additional constraints over the distribution $p$. Suppose $q$ is a prior distribution,~\eqref{example-3} becomes minimizing Kullback-Leiber divergence. We can reformulate~\eqref{example-3} as 
\begin{align}
	\begin{aligned}
		\min \quad & t \\
		\text{s.t.} \quad & \sum_{i \in [d]} p_i = 1, \\
		& (t,q,p) \in \mathcal{K}_{\mathrm{rel}}, \\
		& p \in \mathcal{P}, q \in \mathcal{Q},
	\end{aligned}\label{example-3-extend}
\end{align}
where $\mathcal{K}_{\mathrm{rel}}$ is the relative entropy cone defined in~\eqref{entropy-cone}.
\ifpreprint
\begin{table}[H]
	\begin{center}
		\begin{tabular}{ |c||c|c|c|c|  }
			\hline
			d & \texttt{Clarabel-Entropy} & \texttt{Clarabel-Exp} & \texttt{Mosek} & \texttt{ECOS}\\
			\hline
			1000   & 21.0 & 38.9 & 15 & 34.0 \\
			\hline
			2000   & 51.1 & 77.6 & 31 & 66.2 \\
			\hline
			5000  & 139 & 220 & 79 & 189.1 \\
			\hline
		\end{tabular}
	\end{center}
	\caption{Time (ms) w.r.t. dimensions (relative entropy cones)\label{table-3-1}}
\end{table}
\begin{table}[H]
	\begin{center}
		\begin{tabular}{ |c||c|c|c|c|  }
			\hline
			d & \texttt{Clarabel-Entropy} & \texttt{Clarabel-Exp} & \texttt{Mosek} & \texttt{ECOS}\\
			\hline
			1000   & 16 & 17 & 10 & 18 \\
			\hline
			2000   & 16 & 16 & 10 & 18 \\
			\hline
			5000  & 16 & 17 & 10 & 20 \\
			\hline
		\end{tabular}
	\end{center}
	\caption{Iteration number w.r.t. dimensions (relative entropy cones)\label{table-3-2}}
\end{table}
\fi

If we set $\mathcal{Q}$ to a randomly generated distribution and $\mathcal{P} = \emptyset$, the results are shown in Table~\ref{table-3-1} and Table~\ref{table-3-2}. We found the sparse exploitation can save computational time but not as much as the implementation for generalized power cones. \texttt{Clarabel-Entropy} is about $30\%$ faster than \texttt{Clarabel-Exp} or \texttt{ECOS} but is slower than \texttt{Mosek} which uses the primal-dual scaling strategy with higher-order correction. 
This could be potentially explained by: (1) Reformulating a relative entropy cone as a product of three-dimensional exponential cones doesn't substantially increase the number of nonzero entries in matrix factorization; (2) the self-concordant parameter decreases from $3d+1$ to $3d$, which doesn't reduce as much as the generalized power cone; (3) The initial point for relative entropy cones is approximated and needs to be improved.

\section{Conclusion}\label{section-conclusion}
In this paper, we have presented an efficient sparse decomposition of generalized power cones, power mean cones and relative entropy cones for the dual scaling interior point method. We have proposed logarithmically-homogeneous self-concordant barrier functions for the dual cones of three nonsymmetric cones respectively. We have detailed the computation of corresponding gradients and Hessians and then augmented the linear system to a sparse form by exploiting the low-rank property inside Hessian matrices of these barrier functions, giving a set of decomposition that ensures the quasidefiniteness of the augmented linear system and makes it factorizable under the LDL factorization with static pivoting. The experimental results demonstrated numerical stability of our sparse decomposition for these cones and the efficiency of our method over the state-of-art solvers when an optimization problem is sparse.

\newpage
\bibliography{reference}

\newpage
\appendix

\section{Proof of Theorem~\ref{theorem-barrier}}\label{appendix-proof}
\begin{proof}
	We split the proof into three parts corresponding to each cone mentioned in Theorem~\ref{theorem-barrier}.
	\begin{enumerate}[i)]
		\item $\mathcal{K}_{\text {gpow }(\alpha,d_1,d_2)}$ has a valid $(d_1+1)$-LHSCB function (Theorem 1 in~\cite{Roy22}),
		\begin{align*}
			f_{\text{gpow}}(u, w)=-\ln \left(\prod_{i \in [d_1]} \left(u_i\right)^{2 \alpha_i}- \|w\|^2\right)-\sum_{i \in [d_1]}\left(1-\alpha_i \right) \ln \left(u_i\right),
		\end{align*}
		which implies
		$$
		\vert \langle f'''_{\text{gpow}}(u, w)[h]h,h \rangle \vert \le 2\langle f''_{\text{gpow}}(u, w)h,h \rangle^{3/2}, \ \forall h \in \mathbb{R}^{d_1+d_2}.
		$$
		Note that $f^*_{\text{gpow}}(u, w) = f_{\text{gpow}}(\alpha^{-1}u, w)$ where $\alpha^{-1}u := (\cdots,u_i/\alpha_i,\cdots)$ denotes element-wise division. We then have 
		\begin{align*}
			\vert \langle {f^*_{\text{gpow}}}'''(u, w)[h]h,h \rangle \vert & = \vert \langle f'''_{\text{gpow}}(\alpha^{-1}u, w)[D_1h]D_1h,D_1h \rangle \vert \\
			&\le 2\langle f''_{\text{gpow}}(\alpha^{-1}u, w)D_1h,D_1h \rangle^{3/2} \\
			& = 2\langle {f^*_{\text{gpow}}   }''(\alpha^{-1}u, w)h,h \rangle^{3/2} 
		\end{align*}
		for any $h \in \mathbb{R}^{d_1+d_2}$ with a linear transformation $$D_1 := \text{diag}(1/\alpha_1,\cdots,1/\alpha_i,\cdots,1/\alpha_{d_1},\mathbf{1}^{d_2}).$$ The conjugate barrier $f_{\text{gpow}}^*(u, w)$ satisfies
		\begin{align*}
			f_{\text{gpow}}^*(t u, t w) = f_{\text{gpow}}^*(u, w) - (d_1+1)\ln(t), \ \forall (u,w) \in \mathcal{K}^*_{\text {gpow }(\alpha,d_1,d_2)}, t > 0,
		\end{align*}
		hence $f_{\text{gpow}}^*(u, w)$ is a $(d_1+1)$-LHSCB function of $\mathcal{K}^*_{\text {gpow }(\alpha,d_1,d_2)}$ due to~\eqref{self-concordant}. 
		\bigskip
		\item The same reasoning applies to $f^*_{\mathrm{powm}}$ since we already know  that 
		\begin{align*}
			f(u, w)=-\ln \left(\prod_{i \in [d]} u_i^{\alpha_i}- w\right)-\sum_{i \in [d]} (1-\alpha_i)\ln \left(u_i\right) - \ln(w) 
		\end{align*}
		is a $(d+1)$-LHSCB function of the nonnegative power cone
		\begin{align*}
			\mathcal{K}_{\alpha}^{(n,+)} := \{(u,w) \in \mathbb{R}^d_+ \times \mathbb{R}_+: \prod_{i \in[d]} u_i^{\alpha_i} \ge w\}
		\end{align*} 
		from Theorem 2 in~\cite{Roy22}, which differs from $f^*_{\mathrm{powm}}(u, w)$ by the transformation $D_2 := \text{diag}(1/\alpha_1,\cdots,1/\alpha_i,\cdots,1/\alpha_d,-1)$, like its counterparts for the generalized power cones.
		\bigskip
		\item 	For the proof w.r.t.  $f_{\mathrm{rel}}^*(u,v,w)$, we define the map $A: \mathbb{R}^2 \to \mathbb{R}$
		\begin{align*}
			A(x,y) := -x\ln\left(\frac{x}{y}\right) + x.
		\end{align*}
		Following the same argument for relative entropy cones in~\cite[Appendix E]{DDS}, we have $A(x,y)$ is $(\mathbb{R}_+, 1)$-compatible with the barrier $-\ln(x) - \ln(y)$ and then $t + A(x,y)$ is $(\mathbb{R}_+, 1)$-compatible with the barrier $-\ln(x) - \ln(y)$. Hence $-\ln(t+A(x,y))-\ln(x) - \ln(y)$ is a $3$-self-concordant barrier due to Proposition 5.1.7 in~\cite{Nesterov94}. Moreover, it is easy to verify $-\ln(t+A(x,y))-\ln(x) - \ln(y)$ is $3$-logarithmically homogeneous and is therefore a $3$-LHSCB function. 
		
		This implies that $-\left(w_i - u\ln\left(\frac{u}{v_i}\right)+u\right) - \ln(u) - \ln(v_i)$ is a $3$-self-concordant barrier by replacing $t,x,y$ with $w_i,u,v_i, \forall i \in [d]$. According to the summation rule (Proposition 5.1.3) in~\cite{Nesterov94}, we can prove $f_{\mathrm{rel}}^*(u,v,w)$ is $3d$-LHSCB.
	\end{enumerate}	
\end{proof}

Note that the LHSCB parameter for a relative entropy cone is $\nu = (2d+1)$ in~\cite{DDS,Hypatia}, which is smaller than our choice $\nu = 3d$ for the dual of relative entropy cone. A LHSCB function with a smaller $\nu$ is possible and can reduce the total number of iterations of an IPM in theory, but it is out of the scope of this paper. 

\section{Root finding for power mean cones}\label{appendix-power-mean}
Given $r < 0, x> 0$ and $\alpha_i \in (0,1), \forall i \in [d]$, the Hessian of~\eqref{h-def-power-mean} is
\begin{align*}
	\begin{footnotesize}
\begin{aligned}
	h''(x) &= \sum_{i \in [d]}\alpha_i \cdot \frac{\frac{2r^2}{(r^2+4x^2)^{3/2}}\left(x/\alpha_i + \left(r+\sqrt{r^2+4x^2}\right)/2\right)- \left(1/\alpha_i + 2x/\sqrt{r^2+4x^2}\right)^2}{\left(x/\alpha_i+\left(r+\sqrt{r^2+4x^2}\right)/2\right)^2}\\
	& = \sum_{i \in [d]}\alpha_i \cdot \frac{2r^2x/\alpha_i + r^3+r^2\sqrt{r^2+4x^2}- (r^2+4x^2)^{3/2}/\alpha_i^2 - 4x(r^2+4x^2)/\alpha_i - 4x^2\sqrt{r^2+4x^2}}{(r^2+4x^2)^{3/2}\left(x/\alpha_i+\left(r+\sqrt{r^2+4x^2}\right)/2\right)^2}\\
	& = \sum_{i \in [d]}\alpha_i \cdot \frac{r^3+\sqrt{r^2+4x^2}\left(r^2 -(r^2+4x^2)/\alpha_i^2\right) - 2xr^2/\alpha_i-16x^3/\alpha_i - 4x^2\sqrt{r^2+4x^2}}{(r^2+4x^2)^{3/2}\left(x/\alpha_i+\left(r+\sqrt{r^2+4x^2}\right)/2\right)^2}\\
	& < \sum_{i \in [d]}\alpha_i \cdot \frac{r^3 - 2xr^2/\alpha_i-16x^3/\alpha_i - 4x^2\sqrt{r^2+4x^2}}{(r^2+4x^2)^{3/2}\left(x/\alpha_i+\left(r+\sqrt{r^2+4x^2}\right)/2\right)^2} < 0.
\end{aligned}
\end{footnotesize}
\end{align*}
\end{document}